\documentclass[english]{scrartcl}

\usepackage[a4paper,left=2.9cm, right=2.9cm, top=2.9cm, bottom=2.9cm]{geometry}
\usepackage[english]{babel}
\usepackage{amsmath}
\usepackage{amssymb}
\usepackage{enumitem}
\usepackage[colorlinks = true, linkcolor = blue, citecolor = blue]{hyperref}
\usepackage{amsthm}
\usepackage{cite}
\usepackage[nameinlink]{cleveref}
\usepackage{tikz}
\usepackage{graphicx}
\usepackage{caption}

\addto\captionsenglish{}

\newtheorem{theorem}{Theorem}[section]
\newtheorem{definition}[theorem]{Definition}
\newtheorem{remark}[theorem]{Remark}
\newtheorem{proposition}[theorem]{Proposition}
\newtheorem{lemma}[theorem]{Lemma}

\newtheorem{example}[theorem]{Example}

\Crefname{theorem}{Theorem}{Theorems}
\Crefname{definition}{Definition}{Definitions}
\Crefname{remark}{Remark}{Remarks}
\Crefname{proposition}{Proposition}{Propositions}
\Crefname{lemma}{Lemma}{Lemmas}
\Crefname{corollary}{Corollary}{Corollaries}
\Crefname{example}{Example}{Examples}
\crefname{equation}{}{}
\crefname{figure}{Figure}{Figures}

\title{\LARGE{Formation Shape Control Based on Distance Measurements Using Lie Bracket Approximations}\footnote{This work was funded by the DAAD with funds of the German Federal  Ministry  of  Education  and  Research  (BMBF),  by  the DAAD-Go8  German-Australian  Collaboration  Project,  and  by  the Australian  Research  Council  under  grant  DP160104500.}}

\author{
  Raik Suttner\thanks{R. Suttner is with the Institute of Mathematics, University of Wuerzburg, Wuerzburg, Germany (\texttt{raik.suttner@mathematik.uni-wuerzburg.de}).}
  \and 
    Zhiyong Sun\thanks{Z. Sun was with Research School of Engineering, Australian National University, Canberra, Australia. He is now with the Department of Automatic Control, Lund University, Sweden. (\texttt{sun.zhiyong.} \texttt{cn@gmail.com}, \texttt{zhiyong.sun@control.lth.se}).}
}
\date{}
\begin{document}

\maketitle

\vspace{-1cm}

\begin{abstract}
We study the problem of distance-based formation control in autonomous multi-agent systems in which only distance measurements are available. This means that the target formations as well as the sensed variables are both determined by distances. We propose a fully distributed  distance-only control law, which requires neither a time synchronization of the agents nor storage of measured data. The approach is applicable to point agents in the Euclidean space of arbitrary dimension. Under the assumption of infinitesimal rigidity of the target formations, we show that the proposed control law induces local uniform asymptotic stability. Our approach involves sinusoidal perturbations in order to extract information about the negative gradient direction of each agent's local potential function. An averaging analysis reveals that the gradient information originates from an approximation of Lie brackets of certain vector fields. The method is based on a recently introduced approach to the problem of extremum seeking control. We discuss the relation in the paper.
\end{abstract}


%--------------------------------------------------------------------------------------------------------------
%--------------------------------------------------------------------------------------------------------------
%--------------------------------------------------------------------------------------------------------------

\section{Introduction}
Distance-based formation control is an extensively studied subject in the field of autonomous multi-agent systems. The wish to achieve and maintain prescribed distances among autonomous agents in a distributed way arises in various applications such as leader-follower systems or in the context of formation shape control~\cite{Oh2015}. This task becomes especially difficult if the agents can measure only distances to other members of the team but not their relative positions.

In the present paper, we focus on the model of kinematic points in the Euclidean space of arbitrary dimension. The interaction topology is described by an undirected graph, where each node represents one of the agents. When we connect the current positions of the agents by line segments according to the edges of the graph, we obtain a graph in the Euclidean space, which is also referred to as a formation. We study the problem of distance-based formation control, i.e., the target formations are defined by distances. To be more precise, a target formation is reached if for each edge of the graph, the distance between the corresponding pair of agents is equal to a desired value. These distances are the actively controlled variables. The aim is to find a distributed  control law that steers the agents into one of the target formations. The agents have to accomplish this goal without any shared information like a global coordinate system or a common clock to synchronize their motion.

A well-established approach to solve the above problem  is a gradient descent control law~\cite{Krick2009,Dorfler2009,Oh2014,Mou2016,Sun2016}. For this purpose, every agent is assigned with a local potential function. These functions penalize deviations of the distances to the prescribed values. Each local potential function is defined in such a way that it attains its global minimum value if and only if the distances to the neighbors are equal to the desired values. Thus, a target formation is reached if all agents have minimized the values of their local potential functions. To reach the minimum, every agent follows the negative gradient direction of its local potential function. It is shown in~\cite{Krick2009,Dorfler2009,Oh2014} that this approach can lead to local uniform asymptotic stability with respect to the set of desired states. In fact, by imposing suitable rigidity assumptions on the target formations, one can prove local exponential stability; see, e.g.,~\cite{Mou2016,Sun2016}.

An implementation of the gradient descent control law requires that all agents should be able to measure the \textit{relative positions} to their neighbors in the underlying graph. It is clear that relative positions contain much more information than distances. In other words, the sensed variables are stronger than the controlled variables. It is therefore natural to ask whether distance-based formation control is still possible even if the sensed variables coincide with the controlled variables. This means that each agent can only use its own real-time distance measurements to steer itself into a target formation. We also remark that distance sensing and measurement  has emerged as a mature technique through the development of many low-cost, high precision sensors, such as ultrasonic sensors or laser scanners (see e.g., the survey in~\cite{DorfBook}). Therefore, it motivates us to explore feasible solutions to formation control with distance-only measurement, which also finds significant applications in relevant areas, e.g., multi-robotic coordination in practice.

To our best knowledge, there are just a few studies on formation control by distance-only measurements. The idea in~\cite{Anderson2011} is to compute relative positions directly from distance measurements. However, in order to do so, the agents need more information than just the distances to their neighbors in the underlying graph. It is shown in~\cite{Anderson2011} that if the graph is rigid, and if each agent also has access to the distances to its two-hop neighbors, then they can compute the relative positions by means of a Cholesky factorization of a suitable matrix, which is obtained from distance measurements. Since this factorization is only unique up to an orthogonal transformation, each agent also has to harmonize these relative positions with its individual coordinate system. This requires a certain ability to sense bearing. Thus, it is not sufficient to sense only the actively controlled distances.

Another approach is presented in~\cite{Cao20112}. In contrast to the above strategy, it suffices that each agent measures the distances to its neighbors in the underlying graph. The multi-agent system is divided into subgroups. Following a prescribed schedule, only one of these subgroups is active at a time while the other agents remain at their positions. This requires that the agents share a common clock. It is assumed that the agents of the currently active group have the ability to first localize the resting neighbors of the team by means of distance measurements, and then move into the best possible position. Note that the strategy requires that each agent can map and memorize its own motion within its own local coordinate system. For a minimally rigid graph in the plane, this algorithm leads locally to the desired convergence. However, a generalization to higher dimensions is limited, since the strategy requires a minimally rigid graph that can be constructed by means of a so-called Henneberg sequence~\cite{Anderson2008}, which is, in general, possible only in two dimensions.

A recent attempt to control formation shapes by distance-only measurements can be found in~\cite{Jiang2017}. In this case, the agents perform suitable circular motions with commensurate frequencies. Using collected data from distance measurements during a prescribed time interval, each agent can extract relative positions and relative velocities of its neighbors by means of Fourier analysis. As in~\cite{Cao20112}, the approach in~\cite{Jiang2017} relies on the assumption that the agents share a precise common clock to synchronize
their motions. The proposed strategy leads to convergence if certain control parameters are chosen sufficiently small. However, only existence of these parameters can be ensured but there is no explicit rule how to obtain them. Moreover, the control law only induces convergence to the set of desired formations but not convergence to a single static formation. In general, a common drift of the multi-agent system remains. An extension to higher dimensions is not obvious, since the extraction of relative positions and velocities relies on the geometry of the plane.

A common feature of all of the above strategies is that the agents should be able to compute or infer relative positions from distance measurements. In the present paper, we use a different approach. To explain the idea, we return to the gradient descent control law. In this case, each agent tries to  minimize its own local potential function by moving into the negative gradient direction. A computation of the gradient requires measurements of relative positions. However, the value of each local potential function can be computed from individual distance measurements, and is therefore accessible to every agent. This leads to the question of whether an agent can find the minimum of its local potential function when only the values of the function are available. To solve this problem, we use an approach that was recently introduced in the context of extremum seeking control, see, e.g.,~\cite{Duerr2013,Duerr2014,Duerr2015,Scheinker2013,Scheinker20132,Scheinker2014}. By feeding in suitable sinusoidal perturbation, we induce that the agents are driven, at least approximately, into descent directions of their local potential functions. On average, this leads to a decay of all local potential functions, and therefore convergence to a target formation. The proposed control law for each agent needs no other information than the current value of the local potential function. Under the assumption that the target formations are infinitesimally rigid (see \Cref{sec:basics} for the definition), we can ensure local uniform asymptotic stability. Our control strategy is fully distributed, and can be applied to point agents in any finite dimension.

An earlier attempt to apply Lie bracket approximations to the problem of formation shape control can be found in~\cite{Suttner2017,Suttner2018}. The control law therein requires a permanent all-to-all communication between the agents for an exchange of distance information. The control law in the present paper is based on individual distance measurements and works without any exchange of measured data. Moreover, the results in~\cite{Suttner2017,Suttner2018} contain an unknown frequency parameter for the sinusoidal perturbations. It is assumed that the frequency parameter is chosen sufficiently large; otherwise convergence to a desired formation cannot be guaranteed. The results in the above papers provide only the existence of a sufficiently large frequency parameter, but there is no explicit rule on how to find that value. The control law in the present paper can lead to local uniform asymptotic stability even if the frequency parameter is chosen arbitrarily small. We discuss the influence of the frequency parameter on the performance of our control law in the main part. 

The idea of using Lie bracket approximations to extract directional information from distance measurements can also be found in several other studies. The range of applications includes, among others, multi-agent source seeking~\cite{Duerr20112}, synchronization~\cite{Duerr20132}, and obstacle avoidance~\cite{Duerr20133}. As in the present paper, the desired states are characterized by minima of (artificial) potential functions. A purely distance-based control law is derived by using Lie bracket approximations for the direction of steepest decent. However, the above studies only guarantee practical asymptotic stability, and depend on the unknown frequency parameter that we mentioned in the previous paragraph. Our results for formation shape control are stronger because they ensure local asymptotic stability without the dependence on the frequency parameter. Thus, our findings might also be of interest to the above fields of applications.

The paper is organized as follows. In \Cref{sec:basics}, we introduce basic definitions and notations, which we use throughout the paper. As indicated above, our approach involves the notion of infinitesimal rigidity, which is recalled in \Cref{sec:infRigAndGradEst}. We also derive suitable estimates for the derivatives of the potential functions in this section. The distance-only control law and the main stability result are presented in \Cref{sec:formationControl}, which are supported by certain numerical simulations in the same section. A detailed analysis of the closed-loop system and the proof of the main theorem is carried out in \Cref{sec:proof}. In \Cref{sec:discussion}, we compare the proposed control strategy to the approach in the papers on extremum seeking control that we cited above. The paper ends with some concluding remarks in \Cref{sec:conclusions}.

\section{Basic definitions and notation}\label{sec:basics}
Recall that an \emph{affine Euclidean space} consists of a nonempty set~$P$, a vector space~$V$ with an inner product $\langle\cdot,\cdot\rangle\colon{V}\times{V}\to\mathbb{R}$, and a map $+\colon{P}\times{V}\to{P}$ such that the following conditions are satisfied: (i) $p+0=p$ for every $p\in{P}$, (ii) $(p+v)+w=p+(v+w)$ for all $p\in{P}$, $v,w\in{V}$, and (iii) for any two $p,q\in{P}$, there exists a unique $v\in{V}$, usually denoted by $v=q-p$, such that $p+v=q$. The elements of~$P$ are called \emph{points}, and the elements of~$V$ are called \emph{translations}. For instance $p,q\in{P}$ could be the positions of two agents, and $q-p\in{V}$ is the corresponding translation. In this paper, we consider the particular case $P=V=\mathbb{R}^n$, and $\langle{v},{w}\rangle$ is the standard \emph{Euclidean inner product} of $v,w\in\mathbb{R}^n$. To distinguish~$P$ and~$V$ in our notation, we use letters like $p,q,x$ for points, and letters like~$v,w$ for translations. Throughout the paper, we measure the length of a translation $v\in\mathbb{R}^n$ by the \emph{Euclidean norm} $\|v\|:=\sqrt{\langle{v,v}\rangle}$. Let $\alpha\colon\mathbb{R}^{n}\to\mathbb{R}^m$ be a linear map. Then we usually write~$\alpha{v}$ instead of~$\alpha(v)$ for $v\in{V}$. The \emph{adjoint} of~$\alpha$ is the unique linear map $\alpha^{\top}\colon\mathbb{R}^{m}\to\mathbb{R}^n$ that satisfies $\langle{v,\alpha^{\top}w}\rangle=\langle\alpha{v},w\rangle$ for every $v\in\mathbb{R}^{n}$ and every $w\in\mathbb{R}^{m}$. The \emph{rank} of~$\alpha$, i.e., the dimension of the image of $\alpha$, is denoted by~$\operatorname{rank}\alpha$.

Let $f\colon{U}\to\mathbb{R}^m$ be a map defined on a subset~$U$ of~$\mathbb{R}^n$. If $m=1$, then we call~$f$ a \emph{function}, and if $m=n$, then we call~$f$ a \emph{vector field}. For every given $y\in\mathbb{R}^m$, the \emph{fiber of $f$ over $y$}, denoted by $f^{-1}(y)$, is the (possibly empty) set of all $x\in{U}$ with $f(x)=y$. Suppose that~$U$ is open. If~$f$ is differentiable at some $p\in{U}$, then we let $\operatorname{D}\!f(p)\colon\mathbb{R}^{n}\to\mathbb{R}^m$ denote the \emph{derivative} of~$f$ at~$p$. As usual, for a nonnegative integer~$k$, the map~$f$ is said to be \emph{of class~$C^k$} if it is~$k$ times continuously differentiable. The word \emph{smooth} always means of class~$C^{\infty}$. In case of its existence, the $k$th derivative of~$f$ at $p\in{U}$, $k\geq2$, is denoted by $\operatorname{D}\!^kf(p)$, which is a $k$-linear map. If $n=1$, then we also use symbols like $\dot{f},\ddot{f},\ldots$, or $f',f'',\ldots$ for derivatives. Let $\psi\colon{U}\to\mathbb{R}$ be a differentiable function. For every $p\in{U}$, we let $\nabla\psi(p)\in\mathbb{R}^n$ denote the \emph{gradient} of~$\psi$ at~$p$, i.e., the unique vector that satisfies $\langle\nabla\psi(p),v\rangle=\operatorname{D}\!\psi(p)v$ for every $v\in\mathbb{R}^n$. The map $\nabla\psi\colon{U}\to\mathbb{R}^n$ is a vector field. Let $X\colon{U}\to\mathbb{R}^{n}$ be a vector field. For every $p\in{U}$, we define $(X\psi)(p):=\operatorname{D}\!\psi(p)X(p)$. The resulting function $X\psi\colon{U}\to\mathbb{R}$ is called the \emph{Lie derivative} of~$\psi$ along~$X$. If $X,Y\colon{U}\to\mathbb{R}^n$ are differentiable vector fields, then the vector field $[X,Y]\colon{U}\to\mathbb{R}^n$ defined by $[X,Y](p):=\operatorname{D}\!Y(p)X(p)-\operatorname{D}\!X(p)Y(p)$ is called the \emph{Lie bracket} of~$X,Y$.

%--------------------------------------------------------------------------------------------------------------
%--------------------------------------------------------------------------------------------------------------
%--------------------------------------------------------------------------------------------------------------

\section{Infinitesimal rigidity and gradient estimates}\label{sec:infRigAndGradEst}
The considerations in this section require elementary definitions from differential geometry. As in~\cite{MilnorBook}, we extend the notion of smoothness for maps on not necessarily open domains as follows. A map $f\colon{A}\to{B}$ between arbitrary sets $A\subseteq\mathbb{R}^n$ and $B\subseteq\mathbb{R}^m$ is called \emph{smooth} if for each $x\in{A}$, there exist an open neighborhood~$W$ of~$x$ in~$\mathbb{R}^n$ and a smooth map $F\colon{W}\to\mathbb{R}^m$ such that $f(\xi)=F(\xi)$ holds for every $\xi\in{A\cap{W}}$. A subset~$M$ of~$\mathbb{R}^{n}$ is called a \emph{smooth manifold} of dimension~$k$ if for each point $p\in{M}$ there exists a \emph{parametrization} of~$M$ at~$p$, i.e., a homeomorphism $\phi\colon{V}\to{U}$ from an open subset~$V$ of~$\mathbb{R}^k$ onto an open neighborhood~$U$ of~$p$ in~$M$ (where~$M$ is endowed with the subspace topology) such that both~$\phi$ and~$\phi^{-1}$ are smooth. Let $M\subseteq\mathbb{R}^n$ be a smooth manifold of dimension~$k$, and let $\phi\colon{V}\to{U}$ be a parametrization of~$M$ at $p\in{M}$. Let $\operatorname{D}\!\phi(\phi^{-1}(p))\colon\mathbb{R}^k\to\mathbb{R}^n$ denote the derivative of~$\phi$ at $\phi^{-1}(p)$, where~$\phi$ is considered as a map from~$V$ into~$\mathbb{R}^n$. The image of $\operatorname{D}\!\phi(\phi^{-1}(p))$ is a $k$-dimensional subspace of~$\mathbb{R}^n$, which is called the \emph{tangent space} to~$M$ at~$p$. This space does not depend on the particular choice of the parametrization of~$M$ at~$p$; see again~\cite{MilnorBook}.

\subsection{Infinitesimal rigidity}\label{sec:infRig}
In this subsection, we recall several definitions and statements from~\cite{Asimow1978,Asimow1979}.

An (\emph{undirected}) \emph{graph} $G=(V,E)$ is a set $V=\{1,\ldots,N\}$ together with a nonempty set~$E$ of two-element subsets of~$V$. Each element of~$V$ is referred to as a \emph{vertex} of~$G$ and each element of~$E$ is called an \emph{edge} of~$G$. As an abbreviation, we denote an edge $\{i,j\}\in{E}$ simply by~$ij$. A \emph{framework}~$G(p)$ in~$\mathbb{R}^n$ is a graph~$G$ with~$N$ vertices together with a point
\[
p \ = \ (p_1,\ldots,p_N) \, \in \, \mathbb{R}^n\times\cdots\times\mathbb{R}^n \ = \ \mathbb{R}^{nN}.
\]
Note that for a framework~$G(p)$ in~$\mathbb{R}^{n}$, we may have $p_i=p_j$ for $i\neq{j}$.

Consider a graph $G=(V,E)$ with~$N$ vertices and~$M$ edges, that is, $V=\{1,\ldots,N\}$, and~$E$ has~$M$ elements. Order the~$M$ edges of~$G$ in some way and define the \emph{edge map} $f_G\colon\mathbb{R}^{nN}\to\mathbb{R}^{M}$ of~$G$ by
\[
f_G(p) \ := \ (\ldots, \|p_j-p_i\|^2, \ldots)_{ij\in{E}}
\]
for every $p=(p_1,\ldots,p_N)\in\mathbb{R}^{nN}$. Thus, the value of~$f_G$ at any $(p_1,\ldots,p_N)\in\mathbb{R}^{nN}$ is a vector that collects the squared distances $\|p_j-p_i\|^2$ for all edges ${ij\in{E}}$. A point $p\in\mathbb{R}^{nN}$ is said to be a \emph{regular point} of~$f_G$ if the function $\operatorname{rank}\operatorname{D}\!f_G\colon\mathbb{R}^{nN}\to\mathbb{R}$ attains its global maximum value at~$p$. For later references, we state the following result from~\cite{Asimow1978}, which is an easy consequence of the Inverse Function Theorem.
\begin{proposition}\label{thm:IFT}
Let~$G$ be a graph with~$N$ vertices and~$M$ edges. If $p\in\mathbb{R}^{nN}$ is a regular point of~$f_G$, then there exists an open neighborhood~$U$ of~$p$ in~$\mathbb{R}^{nN}$ such that the subset~$f_G(U)$ of~$\mathbb{R}^M$ is a smooth manifold of dimension $\operatorname{rank}\operatorname{D}\!f_G(p)$.
\end{proposition}
The \emph{complete graph} with~$N$ vertices is the graph with~$N$ vertices that has each two-element subset of $\{1,\ldots,N\}$ as an edge.
\begin{definition}
Let~$G$ be a graph with~$N$ vertices, let~$C$ be the complete graph with~$N$ vertices, and let $p\in\mathbb{R}^{nN}$. The framework~$G(p)$ in~$\mathbb{R}^{n}$ is said to be \emph{rigid} if there exists a neighborhood~$U$ of~$p$ in~$\mathbb{R}^{nN}$ such that
\begin{equation}\label{eq:rigdity}
f_{G}^{-1}(f_G(p))\cap{U} \ = \ f_{C}^{-1}(f_C(p))\cap{U}.
\end{equation}
\end{definition}
Thus, a framework $G(p)$ is rigid if and only if for every $q$ sufficiently close to $p$ with ${\|q_j-q_i\|}={\|p_j-p_i\|}$ for every edge $ij$ of $G$, we have in fact ${\|q_j-q_i\|}={\|p_j-p_i\|}$ for all vertices $i,j$ of $G$. Another result from~\cite{Asimow1978} is the following.
\begin{proposition}\label{thm:EuclidenGroup}
Let~$C$ be the complete graph with~$N$ vertices. For every $p\in\mathbb{R}^{nN}$, the subset $f_{C}^{-1}(f_{C}(p))$ of~$\mathbb{R}^{nN}$ is a smooth manifold.
\end{proposition}
The manifold $f_{C}^{-1}(f_{C}(p))$ is actually analytic and one can derive an explicit formula for its dimension; see again~\cite{Asimow1978}. As in~\cite{Asimow1979}, we use the manifold structure of $f_{C}^{-1}(f_{C}(p))$ to define infinitesimal rigidity.
\begin{definition}
A framework~$G(p)$ in~$\mathbb{R}^n$ is \emph{infinitesimally rigid} if the tangent space to $f_{C}^{-1}(f_{C}(p))$ at~$p$ coincides with the kernel of~$\operatorname{D}\!f_G(p)$.
\end{definition}
To make the notion of infinitesimal rigidity more intuitive, we recall a geometric interpretation from~\cite{Gluck1975}. For this purpose, we consider \emph{smooth isometric deformations} of a given framework~$G(p)$, i.e., smooth curves from an open time interval around~$0$ into the set $f_{G}^{-1}(f_G(p))$ passing through~$p$ at time~$0$. By definition, each such curve $\gamma=(\gamma_1,\ldots,\gamma_N)$ preserves the squared distances $\|\gamma_j(t)-\gamma_i(t)\|^2$ for all edges $ij$ of $G$, and we have $f_G(\gamma(t))=f_G(p)$ for every~$t$ in the domain of~$\gamma$. By the chain rule, this implies that the velocity vector $\dot{\gamma}(0)$ of~$\gamma$ at time~$0$ is an element of the kernel of~$\operatorname{D}\!f_G(p)$ (which is termed \textit{rigidity matrix} in the literature of graph rigidity; see e.g.,~\cite{Asimow1979}). This explains why vectors in the kernel of~$\operatorname{D}\!f_G(p)$ are referred to as \emph{infinitesimal isometric perturbations} of~$G(p)$. On the other hand, the tangent space to the smooth manifold $f_{C}^{-1}(f_{C}(p))$ at~$p$ consists of the velocities of all smooth curves in $f_{C}^{-1}(f_{C}(p))$ passing through~$p$. By definition, the curves in $f_{C}^{-1}(f_{C}(p))$ preserve the squared distances for all vertices of~$G$. Thus, infinitesimal rigidity of~$G(p)$ means that, for every smooth curve~$\gamma$ of the form $\gamma(t)=p+tv$ with~$v$ being an infinitesimal isometric perturbations of~$G(p)$, changes of the squared distances $\|\gamma_j(t)-\gamma_i(t)\|^2$ are not detectable around~$t=0$ in \textit{first-order} terms for all vertices~$i,j$ of~$G$.

For our purposes, it is more convenient to characterize the notion of infinitesimal rigidity by the following result from~\cite{Asimow1979}.
\begin{theorem}\label{thm:27112017}
A framework~$G(p)$ in~$\mathbb{R}^n$ is \emph{infinitesimally rigid} if and only if~$p$ is a regular point of~$f_G$ and if~$G(p)$ is rigid.
\end{theorem}
It follows that the notions of rigidity and infinitesimal rigidity coincide at regular points of the edge map. Finally, we note that it is also possible to characterize infinitesimal rigidity of~$G(p)$ in~$\mathbb{R}^n$ by means of an explicit formula for~$\operatorname{rank}\operatorname{D}\!f_G(p)$; see again~\cite{Asimow1979}.

%--------------------------------------------------------------------------------------------------------------
%--------------------------------------------------------------------------------------------------------------

\subsection{Gradient estimates}\label{sec:GradEst}
In this subsection,~$G=(V,E)$ is a graph with~$N$ vertices and~$M$ edges. Let $f_G\colon\mathbb{R}^{nN}\to\mathbb{R}^{M}$ be the edge map of~$G$. For each edge $ij\in{E}$, let~$d_{ij}$ be a nonnegative real number. Define $d:=(d_{ij}^2)_{ij\in{E}}\in\mathbb{R}^{M}$, where the components of~$d$ are ordered in the same way as the components of~$f_G$. Define a nonnegative smooth function $\psi_{G,d}\colon\mathbb{R}^{nN}\to\mathbb{R}$ by
\begin{equation}\label{eq:potentialFunction}
\psi_{G,d}(p) \ := \ \frac{1}{4}\,\|f_{G}(p)-d\|^2 \ = \ \frac{1}{4}\sum_{ij\in{E}}\big(\|p_j-p_i\|^2-d_{ij}^2\big)^2
\end{equation}
for every $p\in\mathbb{R}^{nN}$. This type of function will appear again in the subsequent sections as local and global potential function of a system of~$N$ agents in~$\mathbb{R}^n$. Our aim is to derive boundedness properties for the gradient of~$\psi_{G,d}$. For this purpose, we need the following auxiliary statements.
\begin{lemma}\label{thm:10112017}
Let $g\colon{U}\to\mathbb{R}$ be a nonnegative~$C^2$ function on an open subset~$U$ of~$\mathbb{R}^k$.
\begin{enumerate}[label=(\alph*)]
	\item\label{thm:10112017:A} For every compact subset~$K$ of~$U$, there exists $c_1>0$ such that $\|\nabla{g(x)}\|^2\leq{c_1}\,g(x)$ for every $x\in{K}$.
	\item\label{thm:10112017:B} Suppose that there exists $z\in{U}$ such that $g(z)=0$ and such that the second derivative of~$g$ at~$z$ is positive definite. Then, there exist $c_3>0$ and a neighborhood~$W$ of~$z$ in~$U$ such that $\|\nabla{g(x)}\|^2\geq{c_3}\,g(x)$ for every $x\in{W}$.
\end{enumerate}
\end{lemma}
The above estimates for the gradient can be easily deduced from Taylor's formula. We omit the proof here. For every $r>0$, define the sublevel set
\[
\psi_{G,d}^{-1}(\leq{r}) \ := \ \{p\in\mathbb{R}^{nN} \ | \ \psi_{G,d}(p)\leq{r}\}.
\]
\begin{proposition}\label{thm:11112017}
\begin{enumerate}[label=(\alph*)]
	\item\label{thm:11112017:A} For every $r>0$, there exists $c_1>0$ such that
	\begin{equation}\label{eq:11112017:A}
	\|\nabla\psi_{G,d}(p)\|^2 \ \leq \ c_1\,\psi_{G,d}(p)
	\end{equation}
	for every $p\in\psi_{G,d}^{-1}(\leq{r})$.
	\item\label{thm:11112017:B} For every $r>0$ and every integer $l\geq2$, there exists $c_2>0$ such that
	\begin{equation}\label{eq:11112017:B}
	|\operatorname{D}\!^l\psi_{G,d}(p)(v_1,\ldots,v_l)| \ \leq \ c_2\,\|v_1\|\cdots\|v_l\|
	\end{equation}
	for every $p\in\psi^{-1}(\leq{r})$ and all $v_1,\ldots,v_l\in\mathbb{R}^{nN}$.
	\item\label{thm:11112017:C} Suppose that for each $p\in{f_{G}^{-1}(d)}$, the framework~$G(p)$ is infinitesimally rigid. Then, there exist $r,c_3>0$ such that
	\begin{equation}\label{eq:11112017:C}
	\|\nabla\psi_{G,d}(p)\|^2 \ \geq \ c_3\,\psi_{G,d}(p)
	\end{equation}
	for every $p\in\psi_{G,d}^{-1}(\leq{r})$.
\end{enumerate}
\end{proposition}
\begin{proof}
For the proof, we need some additional facts from differential geometry, which can be found in~\cite{LeeBook}. An isometry of~$\mathbb{R}^{n}$ is a map $T\colon\mathbb{R}^{n}\to\mathbb{R}^{n}$ such that $\|Ty-Tx\|=\|y-x\|$ for all $x,y\in\mathbb{R}^n$. It is known that the set~$\operatorname{E}(n)$ of all isometries of~$\mathbb{R}^{n}$ forms a Lie group, called the \emph{Euclidean group}. For each $T\in\operatorname{E}(n)$, we define $T^N\colon\mathbb{R}^{nN}\to\mathbb{R}^{nN}$ by $T^Np:=(Tp_1,\ldots,Tp_N)$ for every $p=(p_1,\ldots,p_N)\in\mathbb{R}^{nN}$. It is known that the map $\operatorname{E}(n)\times\mathbb{R}^{nN}\to\mathbb{R}^{nN}$, $(T,p)\mapsto{T^Np}$ is a smooth group action of~$\operatorname{E}(n)$ on~$\mathbb{R}^{nN}$.
For every subset~$S$ of~$\mathbb{R}^{nN}$, we let~$S^{\operatorname{E}(n)}$ denote the set of all~$T^Np$ with $p\in{S}$ and $T\in\operatorname{E}(n)$. In particular, for a single point $p\in\mathbb{R}^{nN}$, the set $p^{\operatorname{E}(n)}:=\{p\}^{\operatorname{E}(n)}$ is called the \emph{orbit} of~$p$. The set $\mathbb{R}^{nN}/\operatorname{E}(n)$ of all orbits endowed with the quotient topology is called the \emph{orbit space}. Note that~$\psi_{G,d}$ is \emph{invariant} under the action of~$\operatorname{E}(n)$, i.e., we have $\psi_{G,d}\circ{T^{N}}=\psi_{G,d}$ for every $T\in\operatorname{E}(n)$. It is easy to check that every sublevel set of~$\psi_{G,d}$ can be reduced to a compact set by isometries, i.e., for every $r>0$, there exists a compact subset~$K$ of~$\mathbb{R}^{nN}$ such that $\psi_{G,d}^{-1}(\leq{r})=K^{\operatorname{E}(n)}$.

To prove parts~\ref{thm:11112017:A} and~\ref{thm:11112017:B}, fix an arbitrary $r>0$. Then, there exists a compact subset~$K$ of~$\mathbb{R}^{nN}$ such that $\psi_{G,d}^{-1}(\leq{r})=K^{\operatorname{E}(n)}$. By \Cref{thm:10112017}~\ref{thm:10112017:A}, there exists $c_1>0$ such that~\cref{eq:11112017:A} holds for every $p\in{K}$. Note that the derivative of any $T\in\operatorname{E}(n)$ is an orthogonal transformation and therefore leaves the Euclidean norm invariant. By the chain rule, we obtain $\|(\nabla\psi_{G,d})\circ{T^N}\|=\|\nabla\psi_{G,d}\|$ for every $T\in\operatorname{E}(n)$, which implies that~\cref{eq:11112017:A} holds in fact for every $p\in{K^{\operatorname{E}(n)}}$. Let $l\geq{2}$ be an integer. Since~$\psi_{G,d}$ is smooth, there exists $c_2>0$ such that~\cref{eq:11112017:B} holds for every $p\in{K}$ and all $v_1,\ldots,v_l\in\mathbb{R}^n$. As for the gradient, it follows from the invariance of~$\psi_{G,d}$ under the action of~$\operatorname{E}(n)$, the chain rule, and the invariance of the Euclidean norm under orthogonal transformations that~\cref{eq:11112017:B} holds for every $p\in{K^{\operatorname{E}(n)}}$ and all $v_1,\ldots,v_l\in\mathbb{R}^n$.

For the rest of the proof, we suppose that~$G(q)$ is infinitesimally rigid for every $q\in{f_{G}^{-1}(d)}$. In the first step, we show that for every $q\in{f^{-1}(d)}$, there exist a neighborhood~$W$ of~$q$ in~$\mathbb{R}^{nN}$ and some constant $c_3>0$ such that \cref{eq:11112017:C} holds for every $p\in{W}$. Suppose that $q\in{f_{G}^{-1}(d)}$. By \Cref{thm:IFT,thm:27112017}, there exists an open neighborhood~$U$ of~$q$ in~$\mathbb{R}^{nN}$ such that the subset~$f_G(U)$ of~$\mathbb{R}^M$ is a smooth manifold of dimension $k:=\operatorname{rank}\operatorname{D}\!f_G(q)$. After possibly shrinking~$U$ around~$q$, we can find a parametrization $\phi\colon{V}\to{f_G(U)}$ for the entire manifold~$f_G(U)$. Then, $\bar{f}_G:=(\phi^{-1}\circ{f_G})|_{U}\colon{U}\to{V}$ is a smooth map with $\operatorname{rank}\operatorname{D}\!\bar{f}_G(q)=k$. Define a smooth function $g_d\colon{V}\to\mathbb{R}$ by $g_d(x):=\|\phi(x)-d\|^2/4$ for every $x\in{V}$. Then, the restriction of~$\psi_{G,d}$ to~$U$ equals $g_d\circ\bar{f}_G$, and by the chain rule, we obtain
\[
\nabla\psi_{G,d}(p) \ = \ \operatorname{D}\!\bar{f}_G(p)^{\top}\nabla{g_d(\bar{f}_G(p))}
\]
for every $p\in{U}$, where $\operatorname{D}\!\bar{f}_G(p)^{\top}\colon\mathbb{R}^{k}\to\mathbb{R}^{nN}$ denotes the adjoint of $\operatorname{D}\!\bar{f}_G(p)\colon\mathbb{R}^{nN}\to\mathbb{R}^{k}$ with respect to the Euclidean inner product. Since $p\mapsto{\operatorname{D}\!\bar{f}_G(p)^{\top}}$ is continuous and has full rank~$k$ at~$q$, there exist a neighborhood~$W$ of~$q$ in~$U$ and a constant $c_3'>0$ such that $\|\operatorname{D}\!\bar{f}_G(p)^{\top}v\|\geq{c_3'}\|v\|$ for every $p\in{W}$ and every $v\in\mathbb{R}^{k}$. In particular, this implies
\[
\|\nabla\psi_{G,d}(p)\| \ \geq \ c_3'\,\|\nabla{g_d(\bar{f}_G(p))}\|
\]
for every $p\in{W}$. Using $\phi(z)=d$ at $z:=\bar{f}_G(q)\in{V}$, a direct computation shows that $\operatorname{D}\!^2{g_d(z)}(v,v)=\|\operatorname{D}\!\phi(z)v\|^2/2$ for every $v\in\mathbb{R}^k$. Since $\operatorname{rank}\operatorname{D}\!\phi(z)=k$, it follows that the second derivative of~$g_d$ at~$z$ is positive definite. Because of \Cref{thm:10112017}~\ref{thm:10112017:B}, we can shrink~$W$ sufficiently around~$q$ and find some $c_3''>0$ such that
\[
\|\nabla{g_d(\bar{f}_G(p))}\|^2 \ \geq \ c_3''\,g_d(\bar{f}_G(p)) \ = \ c_{3}''\,\psi_{G,d}(p)
\]
for every $p\in{W}$. Thus, \cref{eq:11112017:C} holds for every $p\in{W}$ with $c_3:=(c_3')^2\,c_3''$.

Let $\pi\colon\mathbb{R}^{nN}\to\mathbb{R}^{nN}/\operatorname{E}(n)$ be the projection onto the orbit space. Let~$C$ be the complete graph with~$N$ vertices. Note that the edge maps~$f_C$ and~$f_G$ are continuous, and also invariant under the action of~$\operatorname{E}(n)$, i.e., we have $f_C\circ{T^{N}}=f_C$ and $f_G\circ{T^{N}}=f_G$ for every $T\in\operatorname{E}(n)$. Thus, there exist unique continuous maps $\tilde{f}_C,\tilde{f}_G\colon\mathbb{R}^{nN}/\operatorname{E}(n)\to\mathbb{R}^M$ such that $f_{C}=\tilde{f}_{C}\circ\pi$ and $f_{G}=\tilde{f}_{G}\circ\pi$ (see~\cite{LeeBook}). The assumption of rigidity means in the orbit space that for every orbit $\tilde{p}\in\tilde{f}_{G}^{-1}(d)$, there exists a neighborhood~$\tilde{U}$ of~$\tilde{p}$ in $\mathbb{R}^{nN}/\operatorname{E}(n)$ such that $\tilde{f}_G^{-1}(d)\cap\tilde{U}=\tilde{f}_C^{-1}(\tilde{f}_C(\tilde{p}))\cap\tilde{U}$. Since $\tilde{f}_G^{-1}(d)$ is compact, and since $\tilde{f}_C^{-1}(\tilde{f}_C(\tilde{p}))=\{\tilde{p}\}$, it follows that $\tilde{f}_G^{-1}(d)$ only consists of finitely many orbits. Thus, there exists a finite set $P\subseteq{f_{G}^{-1}(d)}$ such that $f_{G}^{-1}(d)=P^{\operatorname{E}(n)}$. Since~$P$ is finite, we obtain from the previous paragraph that there exist a neighborhood~$W$ of~$P$ in~$\mathbb{R}^{nN}$ and some constant $c_3>0$ such that \cref{eq:11112017:C} holds for every $p\in{W}$. Since both~$\psi_{G,d}$ and~$\|\nabla\psi_{G,d}\|$ are invariant under the action of~$\operatorname{E}(n)$, we conclude that \cref{eq:11112017:C} holds for every $p\in{W^{\operatorname{E}(n)}}$. The proof is complete, if we can show that there exists $r>0$ such that $\psi_{G,d}^{-1}(\leq{r})\subseteq{W^{\operatorname{E}(n)}}$. Since $\psi_{G,d}\colon\mathbb{R}^{nN}\to\mathbb{R}$ is continuous and invariant under the action of~$\operatorname{E}(n)$, there exists a unique continuous function $\tilde{\psi}_{G,d}\colon\mathbb{R}^{nN}/\operatorname{E}(n)\to\mathbb{R}$ such that $\psi_{G,d}=\tilde{\psi}_{G,d}\circ\pi$. Since the projection map~$\pi$ is open (see~\cite{LeeBook}), the set $\tilde{W}:=\pi(W)$ is a neighborhood of $\tilde{P}:=\pi(P)=\tilde{\psi}_{G,d}^{-1}(0)$ in $\mathbb{R}^{nN}/\operatorname{E}(n)$. Since~$\tilde{\psi}_{G,d}$ is continuous and has compact sublevel sets, there exists a sufficiently small $r>0$ such that $\tilde{\psi}_{G,d}^{-1}(\leq{r})\subseteq\tilde{W}$. Thus, $\psi_{G,d}^{-1}(\leq{r})\subseteq{W^{\operatorname{E}(n)}}$, which completes the proof.
\end{proof}
\begin{remark}\label{thm:nonCompactness}
In general, the noncompact set $\psi_{G,d}^{-1}(0)$ of global minima of~$\psi_{G,d}$ might have a complicated structure. However, the proof of \Cref{thm:11112017} reveals that under the assumption of infinitesimal rigidity, the set $\psi_{G,d}^{-1}(0)$ is simply the union of orbits of finitely many points in~$\mathbb{R}^{nN}$ under action of the Euclidean group. It therefore suffices to consider~$\psi_{G,d}$ in a small neighborhood of a single point of each orbit. A similar strategy is also applied in several other studies on formation shape control (see, e.g.,~\cite{Helmke2014,Mou2016}). The assumption of infinitesimal rigidity allows us to derive the lower bound~\cref{eq:11112017:C} for the gradient of~$\psi_{G,d}$ on a noncompact sublevel set. This estimate will play an important role in the proof of our main result.
\end{remark}

%--------------------------------------------------------------------------------------------------------------
%--------------------------------------------------------------------------------------------------------------
%--------------------------------------------------------------------------------------------------------------

\section{Formation control}\label{sec:formationControl}

\subsection{Problem description}
We consider a system of~$N$ point agents in~$\mathbb{R}^n$. For each $i=1,\ldots,N$, let $b_{i,1},\ldots,b_{i,n}\in\mathbb{R}^{n}$ be an orthonormal basis of~$\mathbb{R}^n$. We assume that the motion of agent $i\in\{1,\ldots,N\}$ is determined by the kinematic equations
\begin{equation}\label{eq:kinematicPoints}
\dot{p}_i \ = \ \sum_{k=1}^{n}u_{i,k}\,b_{i,k},
\end{equation}
where each~$u_{i,k}$ is a real-valued input channel to control the velocity into direction~$b_{i,k}$. The situation is depicted in \Cref{fig:agentSketch}. It is worth to mention that the directions~$b_{i,k}$ do not need to be known for an implementation of the control law that is presented in the next subsection. 
\begin{figure}
\centering\includegraphics{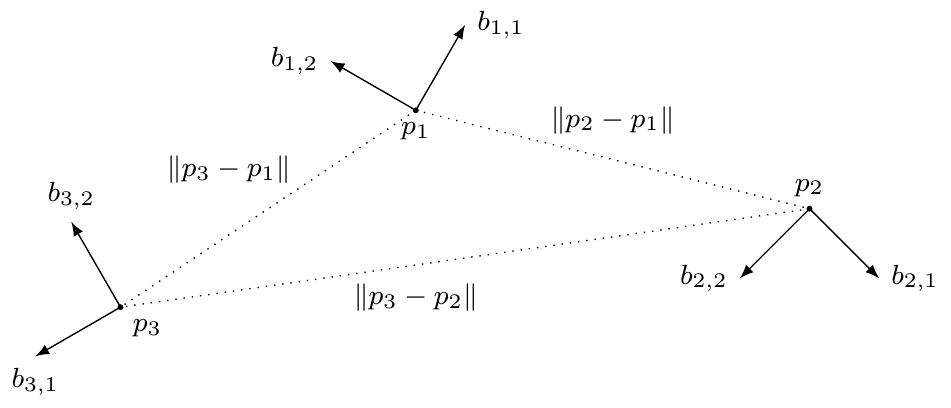}
\caption{A system of $N=3$ point agents in $\mathbb{R}^{n=2}$. Their current distances $\|p_j-p_i\|$ are indicated by dotted lines. The agents do not share information about a global coordinate system. Instead, each agent navigates with respect to its individual body frame, which is defined by the orthonormal velocity directions~$b_{i,k}$.}
\label{fig:agentSketch}
\end{figure}

Suppose that the agents are equipped with very primitive sensors so that they can only measure distances to certain other members of the team. These measurements are described by an (undirected) graph $G=(V,E)$; see \Cref{sec:infRig} for the definition. If there is an edge $ij\in{E}$ between agents $i,j\in{V}$, then it means that agent~$i$ can measure the Euclidean distance $\|p_j-p_i\|$ to agent~$j$ and vice versa. Note that the agents cannot measure relative positions $p_j-p_i$ but only distances. For each edge $ij\in{E}$, let $d_{ij}\geq{0}$ be a nonnegative real number, which is the \emph{desired distance} between agents~$i$ and~$j$. We assume that these distances are \emph{realizable} in~$\mathbb{R}^n$, i.e., there exists $p=(p_1,\ldots,p_N)\in\mathbb{R}^{nN}$ such that $\|p_j-p_i\|=d_{ij}$ for every $ij\in{E}$. We are interested in a distributed and distance-only control law that steers the multi-agent system into such a target formation. The control law that we propose in \Cref{sec:controlLaw} requires only distance measurements and can be implemented directly in each agent's local coordinate frame, which is independent of any global coordinate frame.

We remark that, in the present paper, we assume an undirected graph for modeling a multi-agent formation system, as is often commonly assumed in the literature on multi-agent coordination control (see the surveys~\cite{Oh2015, Cao2013}). This assumption is motivated by various application scenarios. For instance, in practice  agents are often equipped with homogeneous sensors that have the same sensing ability, e.g., same sensing ranges for range sensors. Therefore, it is justifiable to assume bidirectional sensing (described by an undirected graph) in modeling a multi-agent system. Undirected graph also enables a gradient-based control law for stabilizing formation shapes, which may not be possible for general directed graphs. Extensions of the current results to directed graphs will be a topic for future research.

%--------------------------------------------------------------------------------------------------------------
%--------------------------------------------------------------------------------------------------------------

\subsection{Control law and main statement}\label{sec:controlLaw}
For each $i=1,\ldots,N$, define a local potential function $\psi_i\colon\mathbb{R}^{nN}\to\mathbb{R}$ by
\begin{equation}\label{eq:localPotential}
\psi_i(p) \ := \ \frac{1}{4}\sum_{j\in{V}\!\colon\!{ij}\in{E}}\big(\|p_j-p_i\|^2-d_{ij}^2\big)^2
\end{equation}
for every $p=(p_1,\ldots,p_N)\in\mathbb{R}^{nN}$. Note that for the computation of the value of $\psi_i$, agent~$i$ only needs to measure the distances $\|p_j-p_i\|$ to its neighbors $j\in{V}$ with $ij\in{E}$. Choose functions $h_{1},h_{2}\colon\mathbb{R}\to\mathbb{R}$ with the following properties for $\nu=1,2$:
\begin{enumerate}[label=(P\roman*)]
	\item\label{def:hNu:i} $h_{\nu}(y)=0$ for every $y\leq{0}$,
	\item\label{def:hNu:ii} $h_{\nu}$ is bounded and of class~$C^2$ on $(0,\infty)$,
	\item\label{def:hNu:iii} $h_{\nu}(y)/y$ remains bounded as $y\downarrow{0}$,
	\item\label{def:hNu:iv} $h_{\nu}'(y)$ remains bounded as $y\downarrow{0}$,
	\item\label{def:hNu:v} $h_{\nu}''(y)y$ remains bounded as $y\downarrow{0}$,
	\item\label{def:hNu:vi} there exist $r,c>0$ such that
	\begin{equation}\label{eq:hLie}
	[h_1,h_2](y) \ := \ h_2'(y)h_1(y)-h_1'(y)h_2(y) \ \leq \ -c\,{y}
	\end{equation}
	holds for every $y\in(0,r]$,
\end{enumerate}
where $h_{\nu}'$ and $h_{\nu}''$ denote the first and second derivative of~$h_{\nu}$ on $(0,\infty)$, respectively.
\begin{example}\label{exm:hNu}
Let $A\colon[0,\infty)\to\mathbb{R}$ be a bounded function of class~$C^2$ such that $A(0)=0$, and $A'(y)>0$ for every $y\geq{0}$. For instance, $A(y)=\tanh{y}$ or also $A(y)=y/(1+y)$ are two admissible choices. If we define $h_{1}(y):=h_{2}(y):=0$ for $y\leq{0}$ and
\begin{subequations}\label{eq:hNu}
\begin{align}
h_{1}(y) & \ := \ A(y)\,\sin(\log{y}), \\
h_{2}(y) & \ := \ A(y)\,\cos(\log{y})
\end{align}
\end{subequations}
for $y>0$, then a direct computation shows that the functions~$h_1,h_2$ satisfy conditions~\ref{def:hNu:i}-\ref{def:hNu:vi} with $[h_1,h_2](y)=-A(y)^2/y$ for every~$y>0$.
\end{example}
\begin{remark}
The assumptions~\ref{def:hNu:i}-\ref{def:hNu:vi} on~$h_1,h_2$ are imposed to ensure the existence and boundedness of certain Lie derivatives and Lie brackets, which appear later in the analysis of the closed-loop system. These boundedness properties are derived in \Cref{sec:boundednessLemmas}.
\end{remark}

For $i=1,\ldots,N$, and $k=1,\ldots,n$, let~$\omega_{i,k}$ be~$nN$ pairwise distinct positive real numbers, and define $u_{(i,k,1)},u_{(i,k,2)}\colon\mathbb{R}\to\mathbb{R}$ by
\begin{subequations}\label{eq:sinusoids}
\begin{align}
u_{(i,k,1)}(t) & \ := \ \sqrt{\omega_{i,k}}\,\cos(\omega_{i,k}t + \varphi_{i,k}), \\
u_{(i,k,2)}(t) & \ := \ \sqrt{\omega_{i,k}}\,\sin(\omega_{i,k}t + \varphi_{i,k}).
\end{align}
\end{subequations}
with possible phase shifts $\varphi_{i,k}\in\mathbb{R}$.
\begin{example}\label{exm:omegaIK}
Let~$\omega$ be a positive real number, and let
\begin{equation}\label{eq:omegaIK}
\omega_{i,k} \ := \omega\,((i-1)n + k)
\end{equation}
for $i=1,\ldots,N$, and $k=1,\ldots,n$. This defines~$nN$ pairwise distinct positive real numbers~$\omega_{i,k}$.
\end{example}
\begin{remark}
The choice of pairwise distinct frequency coefficients~$\omega_{i,k}$ for the sinusoids~$u_{(i,k,\nu)}$ has the purpose to excite certain Lie brackets of vector fields, which are directly linked to the bracket in~\cref{eq:hLie} of~$h_1,h_2$. This effect is revealed by a suitable averaging analysis in \Cref{sec:averagingLemmas}.
\end{remark}

We propose the control law
\begin{equation}\label{eq:controlLaw}
u_{i,k} \ = \ u_{(i,k,1)}(t)\,h_{1}\big(\psi_i(p)\big) + u_{(i,k,2)}(t)\,h_{2}\big(\psi_i(p)\big)
\end{equation}
for $i=1,\ldots,N$, and $k=1,\ldots,n$.
\begin{remark}
An implementation of the control law~\cref{eq:controlLaw} requires that each agent knows the desired inter-agent distances to its neighbors, and its own pairwise distinct frequencies (and possible phase shifts). Such information can be embedded into the memory of each agent prior to an implementation of the control law. Also, each agent needs to measure the current inter-agent distances (in contrast to relative positions, as assumed in most papers on formation shape control) relative to its neighbors in order to compute the value of its local potential~\cref{eq:localPotential}. The setting of such a control scenario is  common in most distributed control laws, which is acknowledged by the term `centralized design, distributed implementation', which does not contradict with the principle of distributed control (see e.g., the surveys~\cite{Cao2013,Oh2015}). Therefore, the proposed control law is fully distributed.

It is also important to note that we allow arbitrary phase shifts~$\varphi_{i,k}$ in the sinusoids~\cref{eq:sinusoids}. The phase shifts for one agent are not assumed to be known to the other members of the team. In particular, this means that the control law~\cref{eq:controlLaw} requires no time synchronization among the agents. Moreover, since we merely assume that the frequency coefficients~$\omega_{i,k}$ are pairwise distinct, it is not necessary that the sinusoids have a common period.
\end{remark}

It is shown later in \Cref{thm:12112017}~\ref{thm:12112017:A} that for every $i\in\{1,\ldots,N\}$ and every $\nu\in\{1,2\}$, the function $h_{\nu}\circ\psi_i$ is of class~$C^1$. It therefore follows from standard theorems for ordinary differential equations that system~\cref{eq:kinematicPoints} under the control law~\cref{eq:controlLaw} has a unique maximal solution for any initial condition. These solutions do not have a finite escape time because property~\ref{def:hNu:ii} ensures that~\cref{eq:controlLaw} is bounded. In summary, we have the following result.
\begin{proposition}\label{thm:ExistenceUniqueness}
For any initial condition, system~\cref{eq:kinematicPoints} under control law~\cref{eq:controlLaw} has a unique global solution, which we call a \emph{trajectory} of~\cref{eq:kinematicPoints} under~\cref{eq:controlLaw}.
\end{proposition}

To state our main result, we introduce the global potential function $\psi\colon\mathbb{R}^{nN}\to\mathbb{R}$ given by
\begin{equation}\label{eq:gloabalPotential}
\psi(p) \ := \ \frac{1}{4}\,\sum_{ij\in{E}}\big(\|p_j-p_i\|^2-d_{ij}^2\big)^2.
\end{equation}
For every $r>0$, we define the sublevel set
\[
\psi^{-1}(\leq{r}) \ := \ \{ p\in\mathbb{R}^{nN} \ | \ \psi(p)\leq{r} \}.
\]
Note that the zero set of~$\psi$,
\begin{equation}\label{eq:setOfDesiredStates}
\psi^{-1}(0) \ = \ \{ (p_1,\ldots,p_N)\in\mathbb{R}^{nN} \ | \ \forall{ij\in{E}}\colon\|p_j-p_i\|=d_{ij} \},
\end{equation}
is the \emph{set of desired formations}. Since we assume that the distances~$d_{ij}$ are realizable in~$\mathbb{R}^n$, the set \cref{eq:setOfDesiredStates} is not empty.
\begin{theorem}\label{thm:mainResult}
Suppose that for every point~$p$ of~\cref{eq:setOfDesiredStates}, the framework~$G(p)$ is infinitesimally rigid. Then, there exist constants $c,r>0$ such that for every $t_0\in\mathbb{R}$, and every $p_0\in\psi^{-1}(\leq{r})$, the trajectory~$\gamma$ of system~\cref{eq:kinematicPoints} under control law~\cref{eq:controlLaw} with initial condition $\gamma(t_0)=p_0$ converges to some point of~\cref{eq:setOfDesiredStates}, and the estimate
\begin{equation}\label{eq:speedOfConvergence}
\psi(\gamma(t)) \ \leq \ \frac{2\,\psi(p_0)}{1+c\,\psi(p_0)\,(t-t_0)}
\end{equation}
holds for every $t\geq{t_0}$.
\end{theorem}
A detailed proof of \Cref{thm:mainResult} is presented in \Cref{sec:proof}. At this point, we only indicate the reason why the set~\cref{eq:setOfDesiredStates} becomes locally uniformly asymptotically stable for system~\cref{eq:kinematicPoints} under control law~\cref{eq:controlLaw}. Note that the closed-loop system is an ordinary differential equation in the product space~$\mathbb{R}^{nN}$, which consists of the coupled differential equations
\begin{equation}\label{eq:Duerr1}
\dot{p}_i \ = \ \sum_{k=1}^{n}\sum_{\nu=1}^{2}u_{(i,k,\nu)}(t)\,h_{\nu}(\psi_i(p))\,b_{i,k}
\end{equation}
in~$\mathbb{R}^n$ for $i=1,\ldots,N$. One can interpret the right-hand side of~\cref{eq:Duerr1} as a linear combination of the state dependent maps $p\mapsto{h_{\nu}(\psi_i(p))}\,b_{i,k}$ with time-varying coefficient functions~$u_{(i,k,\nu)}$. When we consider the closed-loop system in the product space, each of the maps $p\mapsto{h_{\nu}(\psi_i(p))}\,b_{i,k}$ defines a vector field~$X_{(i,k,\nu)}$ on~$\mathbb{R}^{nN}$. The analysis in \Cref{sec:proof} will show that the trajectories of \cref{eq:Duerr1} are driven into directions of certain Lie brackets of the vector fields~$X_{(i,k,\nu)}$ as long as the system state is sufficiently close to the set~\cref{eq:setOfDesiredStates}. To be more precise, the particular choice of the sinusoids~$u_{(i,k,\nu)}$ with pairwise distinct frequencies~$\omega_{i,k}$ causes the trajectories of~\cref{eq:Duerr1} to follow Lie brackets of the form $[X_{(i,k,1)},X_{(i,k,2)}]$. The ordinary differential equation in~$\mathbb{R}^{nN}$ with the sum of all Lie brackets $\frac{1}{2}[X_{(i,k,1)},X_{(i,k,2)}]$ on the right-hand side is referred to as the corresponding \emph{Lie bracket system}~\cite{Duerr2013}. A direct computation shows that the Lie bracket system is given by the coupled differential equations
\begin{equation}\label{eq:limitSystem}
\dot{p}_i \ = \ \frac{1}{2}\,[h_1,h_2](\psi_i(p))\,\nabla_{p_i}\psi(p)
\end{equation}
in~$\mathbb{R}^n$ for $i=1,\ldots,N$, where $\nabla_{p_i}\psi\colon\mathbb{R}^{nN}\to\mathbb{R}^n$ is the gradient of the global potential function~$\psi$ with respect to the $i$th position vector. Because of property~\ref{def:hNu:vi}, we have $[h_1,h_2](y)<0$ for $y>0$ close to~$0$. Thus, in a neighborhood of~\cref{eq:setOfDesiredStates}, the system state of~\cref{eq:limitSystem} is constantly driven into a descent direction of~$\psi$. The assumption of infinitesimal rigidity ensures that the decay of~$\psi$ along trajectories of~\cref{eq:limitSystem} is sufficiently fast. Since the trajectories of~\cref{eq:Duerr1} approximate the behavior of~\cref{eq:limitSystem} in a neighborhood of~\cref{eq:setOfDesiredStates}, this in turn implies that also the value of~$\psi$ along trajectories of~\cref{eq:kinematicPoints} under~\cref{eq:controlLaw} decays on average. The above strategy is closely related to  several other studies on Lie bracket approximations. We will discuss this relation in \Cref{sec:discussion}.

\begin{remark}\label{rmk:domainOfAttraction}
We emphasize that \Cref{thm:mainResult} guarantees uniform asymptotic stability only in a certain neighborhood of the set~\cref{eq:setOfDesiredStates} of desired formations. The size of the domain of attraction $\psi^{-1}(\leq{r})$ is characterized by the real number ${r>0}$. The value of~$r$ depends on the choice of the functions~$h_{\nu}$ and on the frequency coefficients~$\omega_{i,k}$. As a general rule one can say that the domain of attraction increases if the~$\omega_{i,k}$ are large and also their distances $|\omega_{i,k}-\omega_{i',k'}|$ are large. This property can be ensured by choosing the~$\omega_{i,k}$ as in \Cref{exm:omegaIK} with a large number $\omega>0$. The reader is referred to \Cref{thm:remainderContributions} and to the discussion in \Cref{sec:discussion} for more details. It is an open question whether the domain of attraction of \cref{eq:kinematicPoints} under~\cref{eq:controlLaw} can exceed the domain of attraction of the corresponding Lie bracket system \cref{eq:limitSystem} for a suitable choice of the~$h_{\nu}$ and the~$\omega_{i,k}$. Note that a gradient-based control law can lead to undesired equilibria at stationary points of the potential function.
\end{remark}
\begin{figure}
\centering\includegraphics{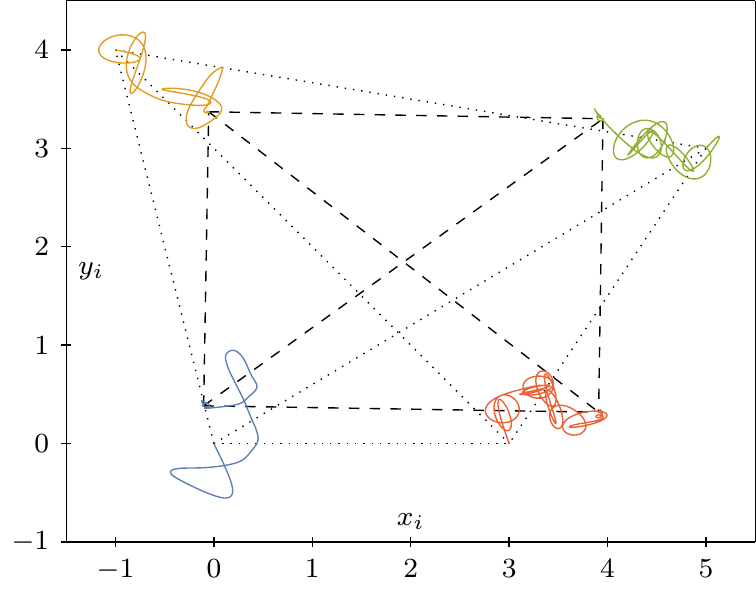}
\caption{Simulation on stabilization control of a four-agent rectangular formation shape. We denote the positions by $p_i=(x_i,y_i)\in\mathbb{R}^2$ for $i=1,\ldots,4$. The initial formation is indicated by dotted lines, and the finial formation is indicated by dashed lines.}
\label{fig:simulation2D}
\end{figure}

\subsection{Simulation examples}\label{sec:simulations}
In this subsection, we provide two simulations to demonstrate the behavior of~\cref{eq:kinematicPoints} under~\cref{eq:controlLaw}. We consider a rectangular formation shape in two dimensions and a double tetrahedron formation shape in three dimensions. One can check that the corresponding frameworks are infinitesimally rigid by means of the rank condition for the derivative of the edge map in~\cite{Asimow1979}. The same formations are also considered in~\cite{Sun2016} for system~\cref{eq:kinematicPoints} under the well-established negative gradient control law. Note that in contrast to the present paper, \textit{relative position measurements} are required in~\cite{Sun2016} to stabilize the desired formation shapes.

Our first example is a system of $N=4$ point agents in the Euclidean space of dimension $n=2$. For $i=1,\ldots,N$, the orthonormal velocity vectors of agent~$i$ in~\cref{eq:kinematicPoints} are given by $b_{i,1}=(\cos\phi_i,\sin\phi_i)$ and $b_{i,2}=(-\sin\phi_i,\cos\phi_i)$, where $\phi_i=i\pi/3$. We let~$G$ be the complete graph of~$N$ nodes. This means that each agent can measure the distances to all other members of the team. The common goal of the agents is to reach a rectangular formation with desired distances $d_{12}=d_{34}=3$, $d_{23}=d_{14}=4$, and $d_{13}=d_{24}=5$. The initial conditions are given by $p_1(0)=(0,0)$, $p_2(0)=(-1,4)$, $p_3(0)=(5,3)$, and $p_4(0)=(3,0)$. As in \Cref{exm:hNu}, we define the functions~$h_1,h_2$ by~\cref{eq:hNu}, where $A:=\tanh$. The frequency coefficients~$\omega_{i,k}$ are chosen as in \Cref{exm:omegaIK} with a positive real number~$\omega$. For the sake of simplicity, the phase shifts~$\varphi_{i,k}$ of the sinusoids are all set equal to zero. It turns out that the initial positions are not in the domain of attraction if we choose $\omega=1$. As indicated in \Cref{rmk:domainOfAttraction}, the domain of attraction becomes larger when we increase~$\omega$. The trajectories for $\omega=7$ are shown in \Cref{fig:simulation2D}.

In the second example, we consider a system of $N=5$ point agents in the Euclidean space of dimension $n=3$. For $i=1,\ldots,N$, the orthonormal velocity vectors of agent~$i$ in~\cref{eq:kinematicPoints} are given by $b_{i,1}=(\sin\theta_i\cos\phi_i,\sin\theta_i\sin\phi_i,\cos\theta_i)$, $b_{i,2}=(-\sin\phi_i,\cos\phi_i,0)$, and $b_{i,3}=(-\cos\theta_i\cos\phi_i,-\cos\theta_i,\sin\theta_i)$, where $\phi_i=i\pi/3$ and $\theta_i=i\pi/6$. We let~$G$ be the graph that originates from the complete graph of~$N$ nodes by removing the edge between the nodes~$4$ and~$5$. The common goal of the agents is to reach a formation shape of a double tetrahedron with desired distances $d_{ij}=2$ for every edge~$ij$ of~$G$. The initial conditions are given by $p_1(0)=(0,-1.0,0.5)$, $p_2(0)=(1.8,1.6,-0.1)$, $p_3(0)=(-0.2,1.8,0.05)$, $p_4(0)=(1.2,1.9,1.7)$ and $p_5(0)=(-1.0,-1.5,-1.2)$. The functions~$h_{\nu}$, the frequency coefficients~$\omega_{i,k}$, and the phase shifts~$\varphi_{i,k}$ are chosen as in the first example. Again, the initial positions are not within the domain of attraction of~\cref{eq:kinematicPoints} under~\cref{eq:controlLaw} for $\omega=1$. However, for $\omega=7$, one can see in \Cref{fig:simulation3D} that the trajectories converge to the desired formation shape.
\begin{figure}
\centering\includegraphics{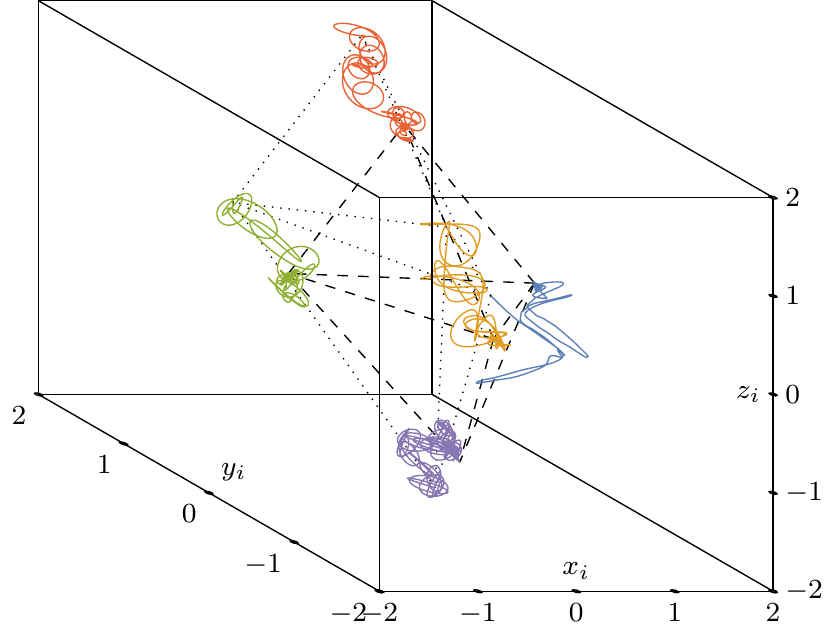}
\caption{Simulation on stabilization control of a double tetrahedron formation. We denote the positions by $p_i=(x_i,y_i,z_i)\in\mathbb{R}^3$ for $i=1,\ldots,5$. The initial formation is indicated by dotted lines, and the finial formation is indicated by dashed lines.}
\label{fig:simulation3D}
\end{figure}

One may interpret the oscillatory trajectories in the simulations as follows. Each agent constantly explores how small changes of its current position influences the value of its local potential function~$\psi_i$. This way an agent obtains gradient information. On average it leads to a decay of all local potential functions. Sufficiently high oscillations are necessary in our approach to ensure that every agent can explore its neighborhood properly. If the value of~$\psi_i$ is small, then the terms $\sin(\log\psi_i)$ and $\cos(\log\psi_i)$ in~\cref{eq:hNu} induce sufficiently high oscillations. When~$\psi_i$ is not small, then an increase of the global frequency parameter~$\omega$ can compensate the lack of oscillations. It is clear that the energy effort to implement~\cref{eq:controlLaw} is much larger than for a gradient-based control law. This is in some sense the price that we have to pay when we reduce the amount of utilized information from the gradient of~$\psi_i$ to the values of~$\psi_i$.

%--------------------------------------------------------------------------------------------------------------
%--------------------------------------------------------------------------------------------------------------
%--------------------------------------------------------------------------------------------------------------

\section{Local asymptotic stability analysis of the closed-loop system}\label{sec:proof}
The aim of this section is to prove \Cref{thm:mainResult}. In the first step, we rewrite system~\cref{eq:kinematicPoints} under control law~\cref{eq:controlLaw} as a control-affine system under open-loop controls. For this purpose, we have to introduce a suitable notation. Recall that, for every $i\in\{1,\ldots,n\}$, the velocity directions $b_{i,1},\ldots,b_{i,n}\in\mathbb{R}^n$ in~\cref{eq:kinematicPoints} are assumed to be an orthonormal basis of~$\mathbb{R}^n$. For each $i\in\{1,\ldots,N\}$ and each $k\in\{1,\ldots,n\}$, define a constant vector field $B_{i,k}\colon\mathbb{R}^{nN}\to\mathbb{R}^{nN}$ by $B_{i,k}(p):=(0,\ldots,0,b_{i,k},0,\ldots,0)$, where $b_{i,k}\in\mathbb{R}^n$ is at the $k$th position. It is clear that the vectors $B_{i,k}(p)$ form an orthonormal basis of~$\mathbb{R}^{nN}$ at any $p\in\mathbb{R}^{nN}$. As an abbreviation, we define an indexing set~$\Lambda$ to be the set of all triples $(i,k,\nu)$ with $i\in\{1,\ldots,N\}$, $k\in\{1,\ldots,n\}$, and $\nu\in\{1,2\}$. For each $m=(i,k,\nu)\in{M}$, define a vector field $X_m\colon\mathbb{R}^{nN}\to\mathbb{R}^{nN}$ by
\begin{equation}\label{eq:newVectorFields}
X_{m}(p) \ := \ h_{\nu}(\psi_i(p))\,B_{i,k}(p).
\end{equation}
When we insert~\cref{eq:controlLaw} into~\cref{eq:kinematicPoints}, the closed-loop system can be written as the control-affine system 
\begin{equation}\label{eq:controlAffineForm}
\dot{p} \ = \ \sum_{m\in\Lambda}u_m(t)\,X_m(p)
\end{equation}
with control vector fields~$X_m$ and open-loop controls~$u_m$.

\subsection{Boundedness properties}\label{sec:boundednessLemmas}
In this subsection, we derive suitable boundedness properties of (iterated) Lie derivatives of the global potential function~$\psi$ along the control vector fields $X_m$ in~\cref{eq:controlAffineForm}. These boundedness properties will ensure in the proof of \Cref{thm:mainResult} in \Cref{sec:proofOfTheorem} that certain remainder terms become small when the agents are close to the set~\cref{eq:setOfDesiredStates} of target formations.

Let~$W_1,W_2$ be subsets of~$\mathbb{R}^k$, and let $W$ be a subset of the (possibly empty) intersection of~$W_1,W_2$. Let $b\colon{W_1}\to\mathbb{R}$ be a nonnegative function. For the sake of convenience, we introduce the following terminology. We say that a function $f\colon{W_2}\to\mathbb{R}$ is \emph{bounded by a multiple of~$b$ on~$W$} if there exists $c>0$ such that $|f(x)|\leq{c\,b(x)}$ for every $x\in{W}$. We say that a vector field $X\colon{W_2}\to\mathbb{R}^k$ is \emph{bounded by a multiple of~$b$ on~$W$} if there exists $c>0$ such that $\|X(x)\|\leq{c\,b(x)}$ for every $x\in{W}$. For a map~$A$ on~$W_2$, which assigns every point of~$W_2$ to a bilinear form $\mathbb{R}^k\times\mathbb{R}^k\to\mathbb{R}$, we say that~$A$ is \emph{bounded by a multiple of~$b$ on~$W$} if there exists $c>0$ such that $|A(x)(v,w)|\leq{c\,b(x)\|v\|\|w\|}$ for every $x\in{W}$ and all $v,w\in\mathbb{R}^k$.

For every $i\in\{1,\ldots,N\}$, and every $r>0$, we define the sublevel set
\[
\psi_i^{-1}(\leq{r}) \ := \ \{ p\in\mathbb{R}^{nN} \ | \ \psi_i(p)\leq{r} \}
\]
where~$\psi_i$ is the local potential function~\cref{eq:localPotential} of agent~$i$. On the other hand, we have defined the global potential function~$\psi$ in~\cref{eq:gloabalPotential} for the entire multi-agent system. It follows directly from the definitions that, for every $i\in\{1,\ldots,N\}$ and every $k\in\{1,\ldots,n\}$, the Lie derivatives of~$\psi_i$ and~$\psi$ along the vector field~$B_{i,k}$ in~\cref{eq:newVectorFields} coincide, i.e., $B_{i,k}\psi=B_{i,k}\psi_i$.
\begin{lemma}\label{thm:12112017}
Let $m=(i,k,\nu)\in\Lambda$ and let $r>0$.
\begin{enumerate}[label=(\alph*)]
	\item\label{thm:12112017:A} The function $h_{\nu}\circ\psi_i$ is of class~$C^1$ and the following boundedness properties hold:
	\begin{enumerate}[label=(\roman*)]
		\item\label{thm:12112017:A:I} $h_{\nu}\circ\psi_i$ is bounded by a multiple of~$\psi_i$ on $\psi_i^{-1}(\leq{r})$;
		\item\label{thm:12112017:A:II} $\nabla(h_{\nu}\circ\psi_i)$ is bounded by a multiple of~$\psi_i^{1/2}$ on $\psi_i^{-1}(\leq{r})$.
	\end{enumerate}
	\item\label{thm:12112017:B} The Lie derivative~$X_m\psi$ of~$\psi$ along~$X_m$ is of class~$C^2$ and the following boundedness properties hold:
	\begin{enumerate}[label=(\roman*)]
		\item\label{thm:12112017:B:I}~$X_m\psi$ is bounded by a multiple of~$\psi_i^{3/2}$ on $\psi_i^{-1}(\leq{r})$;
		\item\label{thm:12112017:B:II} $\nabla(X_m\psi)$ is bounded by a multiple of~$\psi_i$ on $\psi_i^{-1}(\leq{r})$;
		\item\label{thm:12112017:B:III} $\operatorname{D}\!^2(X_m\psi)$ is bounded by a multiple of~$\psi_i^{1/2}$ on $\psi_i^{-1}(\leq{r})$.
	\end{enumerate}
\end{enumerate}
\end{lemma}
\begin{proof}
Let~$Z_i$ be  the zero set $\psi_i^{-1}(0)$ of~$\psi_i$, and let $U_i:=\mathbb{R}^{nN}\setminus{Z_i}$ be the set of points at which~$\psi_i$ is strictly positive. Note that~$\psi_i$ is of the form~\cref{eq:potentialFunction} with respect to the subgraph of~$G$ that originates by restricting~$G$ to the vertex~$i$ and its neighbors in~$G$. Therefore, \Cref{thm:11112017} can be applied to~$\psi_i$. Recall that~$h_{\nu}$ is assumed to satisfy the properties~\ref{def:hNu:i}-\ref{def:hNu:vi}, which are listed in \Cref{sec:controlLaw}.

Because of property~\ref{def:hNu:iii}, the function $h_{\nu}\circ\psi_i$ is bounded by a multiple of~$\psi_i$ on $\psi_i^{-1}(\leq{r})$. It follows that there exists $c>0$ such that
\[
|(h_{\nu}\circ\psi_i)(q)-(h_{\nu}\circ\psi_i)(p)| \ \leq \ c\,|\psi_i(q)-\psi_i(p)|
\]
for every $p\in{Z_i}$, and every $q\in\psi_i^{-1}(\leq{r})$. This implies that the derivative of $h_{\nu}\circ\psi_i$ exists and vanishes at every $p\in{Z_i}$ with vanishing derivative. Since property~\ref{def:hNu:ii} ensures that $h_{\nu}\circ\psi_i$ is of class~$C^2$ on~$U_i$, we can compute
\[
\nabla(h_{\nu}\circ\psi_i)(p)  \ = \ h_{\nu}'(\psi_i(p))\,\nabla\psi_i(p)
\]
for every $p\in{U_i}$. Because of property~\ref{def:hNu:iv}, the function $h_{\nu}'\circ\psi_i\colon{U_i}\to\mathbb{R}$ is bounded by a constant on $U_i\cap\psi_i^{-1}(\leq{r})$. By \Cref{thm:11112017}~\ref{thm:11112017:A}, the vector field~$\nabla\psi_i$ is bounded by a multiple of~$\psi_i^{1/2}$ on $\psi_i^{-1}(\leq{r})$. It follows that $\nabla(h_{\nu}\circ\psi_i)$ is also bounded by a multiple of~$\psi_i^{1/2}$ on $\psi_i^{-1}(\leq{r})$, and that $\nabla(h_{\nu}\circ\psi_i)$ is continuous on~$\mathbb{R}^{nN}$. This proves part~\ref{thm:12112017:A}.

Since $B_{i,k}\psi=B_{i,k}\psi_i$, we have
\[
(X_m\psi)(p) \ = \ (h_{\nu}\circ\psi_i)(p)\,(B_{i,k}\psi_i)(p)
\]
for every $p\in\mathbb{R}^{nN}$. By \Cref{thm:11112017}~\ref{thm:11112017:A}, the function $B_{i,k}\psi_i=\langle\nabla\psi_i,B_{i,k}\rangle$ is bounded by a multiple of~$\psi_i^{1/2}$ on $\psi_i^{-1}(\leq{r})$. Because of part~\ref{thm:12112017:A}, we conclude that~$X_m\psi$ is bounded by a multiple of~$\psi_i^{3/2}$ on $\psi_i^{-1}(\leq{r})$. Moreover, part~\ref{thm:12112017:A} ensures that~$X_m\psi$ is at least of class~$C^1$, and therefore we can compute
\[
\nabla(X_m\psi)(p) \ = \ (B_{i,k}\psi_i)(p)\,\nabla(h_{\nu}\circ\psi_i)(p) + (h_{\nu}\circ\psi_i)(p)\,\nabla(B_{i,k}\psi_i)(p)
\]
for every $p\in\mathbb{R}^{nN}$. We obtain from \Cref{thm:11112017}~\ref{thm:11112017:B} that the vector field $\nabla(B_{i,k}\psi_i)$ is bounded by a constant on $\psi_i^{-1}(\leq{r})$. Using again \Cref{thm:11112017}~\ref{thm:11112017:A} and part~\ref{thm:12112017:A} for the other constituents of~$\nabla(X_m\psi)$, we derive that~$\nabla(X_m\psi)$ is bounded by a multiple of~$\psi_i$ on $\psi_i^{-1}(\leq{r})$. It follows that there exists $c>0$ such that
\[
\|\nabla(X_m\psi)(q)-\nabla(X_m\psi)(p)\| \ \leq \ c\,|\psi_i(q)-\psi_i(p)|
\]
for every $p\in{Z_i}$, and every $q\in\psi_i^{-1}(\leq{r})$. This implies that the derivative of~$\nabla(X_m\psi)$ exists and vanishes at every $p\in{Z_i}$. Since $h_{\nu}\circ\psi_i$ is of class~$C^2$ on~$U_i$, we can compute
\[
\operatorname{D}\!^2(h_{\nu}\circ\psi_i)(p)(v,w)  = (h_{\nu}''\circ\psi_i)(p)\,\langle\nabla\psi_i(p),v\rangle\,\langle\nabla\psi_i(p),w\rangle + (h_{\nu}'\circ\psi_i)(p)\,\operatorname{D}\!^2\psi_i(p)(v,w)
\]
for every $p\in{U_i}$ and all $v,w\in\mathbb{R}^{nN}$. Because of~\ref{def:hNu:iv}, the function $h_{\nu}'\circ\psi_i$ is bounded by a constant on $U_i\cap\psi_i^{-1}(\leq{r})$, and  because of~\ref{def:hNu:v}, the function $(h_{\nu}''\circ\psi_i)\,\psi_i$ is bounded by a constant on $U_i\cap\psi_i^{-1}(\leq{r})$. By \Cref{thm:11112017}~\ref{thm:11112017:A}, the gradient $\nabla\psi_i$ is bounded by a multiple of $\psi_i^{1/2}$ on $\psi_i^{-1}(\leq{r})$. By \Cref{thm:11112017}~\ref{thm:11112017:B}, $\operatorname{D}\!^2\psi_i$ is bounded by a constant on $\psi_i^{-1}(\leq{r})$. It follows that $\operatorname{D}\!^2(h_{\nu}\circ\psi_i)$ is bounded by a constant on $U_i\cap\psi_i^{-1}(\leq{r})$. We compute
\begin{align*}
\operatorname{D}\!^2(X_m\psi)(p)(v,w) & \ = \ \operatorname{D}\!^2(h_{\nu}\circ\psi_i)(p)(v,w)\,(B_{i,k}\psi_i)(p) \\
& \qquad + \langle\nabla(h_{\nu}\circ\psi_i)(p),v\rangle\,\langle\nabla(B_{i,k}\psi_i)(p),w\rangle \\
& \qquad + \langle\nabla(h_{\nu}\circ\psi_i)(p),w\rangle\,\langle\nabla(B_{i,k}\psi_i)(p),v\rangle \\
& \qquad + (h_{\nu}\circ\psi_i)(p)\,\operatorname{D}\!^2(B_{i,k}\psi_i)(p)(v,w)
\end{align*}
for every $p\in{U_i}$ and all $v,w\in\mathbb{R}^{nN}$. We obtain from \Cref{thm:11112017}~\ref{thm:11112017:B} that the map $\operatorname{D}\!^2(B_{i,k}\psi_i)$ is bounded by a constant on $\psi_i^{-1}(\leq{r})$. For the other constituents of $\operatorname{D}\!^2(X_m\psi)$, we already know boundedness properties on $U_i\cap\psi_i^{-1}(\leq{r})$. This way, we conclude that $\operatorname{D}\!^2(X_m\psi)$ is bounded by multiple of~$\psi_i^{1/2}$ on $U_i\cap\psi_i^{-1}(\leq{r})$. Since we already know that the second derivative of $(X_m\psi)$ exists and vanishes on~$Z_i$, it follows that $\operatorname{D}\!^2(X_m\psi)$ exists as a continuous map on~$\mathbb{R}^{nN}$, and that it is bounded by a multiple of~$\psi_i^{1/2}$ on $\psi_i^{-1}(\leq{r})$.
\end{proof}
Note that, for every $i\in\{1,\ldots,N\}$, we have $\psi_i\leq\psi$ on~$\mathbb{R}^{nN}$. This implies that $\psi^{-1}(\leq{r})$ is a subset of $\psi_i^{-1}(\leq{r})$ for every $r>0$ and every $i\in\{1,\ldots,N\}$. In the next step, we use \Cref{thm:12112017} to derive the following result.
\begin{lemma}\label{thm:13112017}
Let $m_{\ell}=(i_{\ell},k_{\ell},\nu_{\ell})\in\Lambda$ for $\ell=1,2,3$ and let $r>0$.
\begin{enumerate}[label=(\alph*)]
	\item\label{thm:13112017:A}
	\begin{enumerate}[label=(\roman*)] 
		\item\label{thm:13112017:A:I}~$X_{m_1}$ is of class~$C^1$ on~$\mathbb{R}^{nN}$, and bounded by a multiple of~$\psi$ on $\psi^{-1}(\leq{r})$.
		\item\label{thm:13112017:A:II} $(\operatorname{D}\!X_{m_1})X_{m_2}$ is of class~$C^0$ on~$\mathbb{R}^{nN}$, and bounded by a multiple of~$\psi^{3/2}$ on $\psi^{-1}(\leq{r})$.
	\end{enumerate}
	\item\label{thm:13112017:B}
	\begin{enumerate}[label=(\roman*)]
		\item\label{thm:13112017:B:I} $X_{m_1}\psi$ is of class~$C^2$ on~$\mathbb{R}^{nN}$, and bounded by a multiple of~$\psi^{3/2}$ on $\psi^{-1}(\leq{r})$.
		\item\label{thm:13112017:B:II} $X_{m_2}(X_{m_1}\psi)$ is of class~$C^1$ on~$\mathbb{R}^{nN}$, and bounded by a multiple of~$\psi^{2}$ on $\psi^{-1}(\leq{r})$.
		\item\label{thm:13112017:B:III} $X_{m_3}(X_{m_2}(X_{m_1}\psi))$ is of class~$C^0$ on~$\mathbb{R}^{nN}$, and bounded by a multiple of~$\psi^{5/2}$ on $\psi^{-1}(\leq{r})$.
	\end{enumerate}
\end{enumerate}
\end{lemma}
\begin{proof}
Because of \Cref{thm:12112017}~\ref{thm:12112017:A}, the vector field $X_{m_1}=(h_{\nu_1}\circ\psi_{i_1})\,B_{i_1,k_1}$ is of class~$C^1$, and it is bounded by a multiple of~$\psi$ on $\psi^{-1}(\leq{r})$. We also obtain from \Cref{thm:12112017}~\ref{thm:12112017:A} that $\nabla(h_{\nu_1}\circ\psi_{i_1})$ is of class~$C^0$ and bounded by a multiple of~$\psi^{3/2}$ on $\psi^{-1}(\leq{r})$. It follows that the same is true for the derivative of~$X_{m_1}$. This implies the second statement of part~\ref{thm:13112017:A}.

To prove part~\ref{thm:13112017:B}, note that by \Cref{thm:12112017}~\ref{thm:12112017:B}, the function~$X_{m_1}\psi$ is of class~$C^2$ and also bounded by a multiple of~$\psi^{3/2}$ on $\psi^{-1}(\leq{r})$. In particular, we can compute the Lie derivatives
\begin{align*}
X_{m_2}(X_{m_1}\psi) & \ = \ (h_{\nu_2}\circ\psi_{i_2})\,(B_{i_2,k_2}(X_{m_1}\psi)), \\
X_{m_3}(X_{m_2}(X_{m_1}\psi)) & \ = \ (h_{\nu_3}\circ\psi_{i_3})\,(B_{i_3,k_3}(h_{\nu_2}\circ\psi_{i_2}))\,(B_{i_2,k_2}(X_{m_1}\psi))  \\
& \qquad + (h_{\nu_3}\circ\psi_{i_3})\,(h_{\nu_2}\circ\psi_{i_2})\,(B_{i_3,k_3}(B_{i_2,k_2}(X_{m_1}\psi))),
\end{align*}
which are of class~$C^1$ and~$C^0$, respectively. The asserted boundedness properties of the functions $X_{m_2}(X_{m_1}\psi)$ and $X_{m_3}(X_{m_2}(X_{m_1}\psi))$ now follow immediately from \Cref{thm:12112017}.
\end{proof}
Because of \Cref{thm:13112017}~\ref{thm:13112017:A}, for every $i=1,\ldots,N$ and every $k=1,\ldots,n$, the Lie bracket $[X_{(i,k,1)},X_{(i,k,2)}]$ of $X_{(i,k,1)}$, $X_{(i,k,2)}$ exists as a continuous vector field on~$\mathbb{R}^{nN}$. Thus,
\begin{equation}\label{eq:Duncan0}
Y \ := \ \frac{1}{2}\sum_{i=1}^{N}\sum_{k=1}^{n}[X_{(i,k,1)},X_{(i,k,2)}] \colon \mathbb{R}^{nN} \to \mathbb{R}^{nN}
\end{equation}
is also a well-defined continuous vector on~$\mathbb{R}^{nN}$. In fact, one can show that~$Y$ is of class~$C^1$, but we do not need this property in the following. Moreover, we define a function $h\colon\mathbb{R}\to\mathbb{R}$ by $h(y):=0$ for $y\leq{0}$, and by
\[
h(y) \ := \ [h_1,h_2](y)
\]
for $y>0$ with $[h_1,h_2](y)$ as in~\cref{eq:hLie}. Using the identity $B_{i,k}\psi=B_{i,k}\psi_i$, a direct computation shows that
\[
[X_{(i,k,1)},X_{(i,k,2)}] \ = \ (h\circ\psi_i)\,(B_{i,k}\psi)\,B_{i,k}
\]
holds on~$\mathbb{R}^{nN}$ for $i=1,\ldots,N$ and $k=1,\ldots,n$. Thus, the vector field~$Y$ is given by
\begin{equation}\label{eq:Duncan1}
Y \ = \ \frac{1}{2}\sum_{i=1}^{N}\sum_{k=1}^{n}(h\circ\psi_i)\,(B_{i,k}\psi)\,B_{i,k}.
\end{equation}
It is now easy to see that the differential equation $\dot{p}=Y(p)$ in~$\mathbb{R}^{nN}$ coincides with the~$N$ coupled differential equations~\cref{eq:limitSystem} in~$\mathbb{R}^{n}$. As indicated earlier, in a neighborhood of the set~\cref{eq:setOfDesiredStates}, the system state of~\cref{eq:limitSystem} is constantly driven into a descent direction of~$\psi$. We make this statement more precise by providing an estimate for the Lie derivative of~$\psi$ along~$Y$:
\begin{lemma}\label{thm:14112017}
There exist $c,r>0$ such that
\[
(Y\psi)(p) \ \leq \ -c\,\|\nabla\psi(p)\|^4
\]
for every $p\in\psi^{-1}(\leq{r})$.
\end{lemma}
\begin{proof}
Since we assume that $h_1,h_2$ satisfy property~\ref{def:hNu:vi} in \Cref{sec:controlLaw}, there exist $c_h,r>0$ such that $h(y)\leq-{c_h\,y}$ for every $y\in[0,r]$. Because of~\cref{eq:Duncan1}, this implies
\[
Y\psi \ \leq \ -c_h\sum_{i=1}^{N}\sum_{k=1}^{n}\psi_i\,(B_{i,k}\psi)^2
\]
on $\psi^{-1}(\leq{r})$. We obtain from \Cref{thm:11112017}~\ref{thm:11112017:A} that for every $i\in\{1,\ldots,N\}$, there exists $c_{i}>0$ such that for every $k\in\{1,\ldots,n\}$, we have
\[ 
\psi_i \ \geq\ c_{i}\,\|\nabla\psi_i\|^2 \ \geq \ c_{i}\,(B_{i,k}\psi_i)^2 \ = \ c_{i}\,(B_{i,k}\psi)^2
\]
on $\psi^{-1}(\leq{r})$. Thus, there exists $\tilde{c}>0$ such that
\[
Y\psi \ \leq \ -\tilde{c}\sum_{i=1}^{N}\sum_{k=1}^{n}(B_{i,k}\psi)^4
\]
on $\psi^{-1}(\leq{r})$. Note that the sum on the right-hand side is the 4th power of the $4$-norm of the vector field with components $B_{i,k}\psi$. On the other hand, we have $\|\nabla\psi\|^2=\sum_{i=1}^{N}\sum_{k=1}^{n}(B_{i,k}\psi)^2$ since the vector fields~$B_{i,k}$ form an orthonormal frame of~$\mathbb{R}^{nN}$. Since all norms on~$\mathbb{R}^{nN}$ are equivalent, the asserted estimate follows.
\end{proof}

\subsection{Averaging}\label{sec:averagingLemmas}
The next step in the analysis of the closed-loop system~\cref{eq:controlAffineForm} addresses the trigonometric functions~$u_m$ therein. Instead of the differential equation~\cref{eq:controlAffineForm}, it is more convenient to consider the corresponding integral equation. Repeated integration by parts on the right-hand side of this integral equation shows that the functions~$u_m$ give rise to an averaged vector field, which consists of Lie brackets of the~$X_m$. A much more general treatment of this averaging procedure is done in~\cite{Kurzweil1987,Kurzweil19882,Kurzweil1988,Sussmann1991,Liu19972,Liu1997}. In the following, we introduce the notation from~\cite{Liu19972,Liu1997}.

For every $m=(i,k,\nu)\in\Lambda$, define two complex constants $\eta_{\pm\omega_{i,k},m}\in\mathbb{C}$ as follows. If $\nu=1$, let $\eta_{\pm\omega_{i,k},m}:=\sqrt{\omega_{i,k}}\,\mathrm{e}^{\pm\mathrm{i}\varphi_{i,k}}/2$, and otherwise, i.e., if $\nu=2$, let $\eta_{\pm\omega_{i,k},m}:=\pm\sqrt{\omega_{i,k}}\,\mathrm{e}^{\pm\mathrm{i}\varphi_{i,k}}/(2\mathrm{i})$, where~$\mathrm{i}$ denotes the imaginary unit. Moreover, let $\Omega(m):=\{\pm\omega_{i,k}\}$. Then, we can write~$u_m$ in~\cref{eq:sinusoids} as
\[
u_m(t) \ = \ \sum_{\omega\in\Omega(m)}\eta_{\omega,m}\,\mathrm{e}^{\mathrm{i}\omega{t}}
\]
for every $t\in\mathbb{R}$. Additionally, define two functions $v_{m},\widetilde{UV}_{m}\colon\mathbb{R}\to\mathbb{R}$ by
\begin{align*}
v_{m}(t) & \ := \ 0, \\
\widetilde{UV}_{m}(t) & \ := \ -\sum_{\omega\in\Omega(m)}\frac{\eta_{\omega,m}}{\mathrm{i}\omega}\,\mathrm{e}^{\mathrm{i}\omega{t}}.
\end{align*}
For all $m,m'\in\Lambda$, define $v_{m',m},\widetilde{UV}_{m',m}\colon\mathbb{R}\to\mathbb{R}$ by
\begin{align*}
v_{m',m}(t) & \ := \ -\sum_{\substack{(\omega',\omega)\in\Omega(m')\times\Omega(m)\\\omega'+\omega=0}}\frac{\eta_{\omega',m'}\,\eta_{\omega,m}}{\mathrm{i}\,\omega}, \\
\widetilde{UV}_{m',m}(t) & \ := \ \sum_{\substack{(\omega',\omega)\in\Omega(m')\times\Omega(m)\\\omega'+\omega\neq{0}}}\frac{\eta_{\omega',m'}\,\eta_{\omega,m}}{\mathrm{i}^2\,\omega(\omega'+\omega)}\,\mathrm{e}^{\mathrm{i}(\omega'+\omega)t}.
\end{align*}
\begin{remark}\label{thm:frequencyEstimates}
Suppose that the frequency coefficients~$\omega_{i,k}$ are given by~\cref{eq:omegaIK} in \Cref{exm:omegaIK}. Then, it follows directly from the definition of the functions~$\widetilde{UV}_{m}$ and~$\widetilde{UV}_{m',m}$ that there exists $c>0$ such that
\[
\big|\widetilde{UV}_{m}(t)\big| \ \leq \ \frac{c}{\sqrt{\omega}} \qquad \text{and} \qquad \big|\widetilde{UV}_{m',m}(t)\big| \leq \ \frac{c}{\omega}
\]
for all $m,m'\in\Lambda$ and every $t\in\mathbb{R}$. This shows that the~$\widetilde{UV}_{m}$ and~$\widetilde{UV}_{m',m}$ converge uniformly to~$0$ as the global frequency parameter~$\omega$ tends to~$\infty$. We will address this convergence property again in \Cref{thm:remainderContributions} and in \Cref{sec:discussion}.
\end{remark}
A direct computation reveals that the above functions are related as follows.
\begin{lemma}\label{thm:15112017}
Let $m_1=(i_1,k_1,\nu_1),m_2=(i_2,k_2,\nu_2)\in\Lambda$ and $t_0,t\in\mathbb{R}$. Then:
\begin{align*}
\int_{t_0}^{t}\big(v_{m_1}(s) - u_{m_1}(s)\big)\,\mathrm{d}s & \ = \ \widetilde{UV}_{m_1}(t) - \widetilde{UV}_{m_1}(t_0), \\
\int_{t_0}^{t}\big(v_{m_2,m_1}(s) - u_{m_2}(s)\,\widetilde{UV}_{m_1}(s)\big)\,\mathrm{d}s & \ = \ \widetilde{UV}_{m_2,m_1}(t) - \widetilde{UV}_{m_2,m_1}(t_0),
\end{align*}
and
\[
v_{m_2,m_1}(t) \ = \ \left\{\begin{tabular}{cl} $+\frac{1}{2}$ & if $(i_2,k_2)=(i_1,k_1)$ and $\nu_2=1$ and $\nu_1=2$, \\ $-\frac{1}{2}$ & if $(i_2,k_2)=(i_1,k_1)$ and $\nu_2=2$ and $\nu_1=1$, \\ $0$ & otherwise. \end{tabular}\right.
\]
\end{lemma}
We omit the proof here, and refer the reader instead to the computations in the proof of the main theorem in~\cite{Liu19972}.

Because of \Cref{thm:15112017}, we have
\begin{equation}\label{eq:Duncan2}
\sum_{m_1,m_2\in\Lambda}v_{m_2,m_1}\,X_{m_2}(X_{m_1}\psi) \ = \ \frac{1}{2}\sum_{i=1}^{N}\sum_{k=1}^{n}([X_{(i,k,1)},X_{(i,k,2)}]\psi)(p) \ = \ Y\psi,
\end{equation}
where the vector field $Y\colon\mathbb{R}^{nN}\to\mathbb{R}^{nN}$ is given by~\cref{eq:Duncan0}. Next, we write down the propagation of~$\psi$ along trajectories of~\cref{eq:controlAffineForm} as an integral equation, which consists of the averaged part~\cref{eq:Duncan2} and a remainder part. Recall that we already know from \Cref{thm:ExistenceUniqueness} that there exists a unique global solution of~\cref{eq:controlAffineForm} for any initial condition.
\begin{proposition}\label{thm:16112017}
Let $\gamma\colon\mathbb{R}\to\mathbb{R}^{nN}$ be a trajectory of~\cref{eq:controlAffineForm}. Then
\begin{subequations}\label{eq:16112017:01}
\begin{align}
\psi(\gamma(t)) & \ = \ \psi(\gamma(t_0)) + \int_{t_0}^{t}(Y\psi)(\gamma(s))\,\mathrm{d}s - (D_1\psi)(t_0,\gamma(t_0)) \label{eq:16112017:01:A} \\
& \qquad + (D_1\psi)(t,\gamma(t)) + \int_{t_0}^{t}(D_2\psi)(s,\gamma(s))\,\mathrm{d}s \label{eq:16112017:01:B}
\end{align}
\end{subequations}
for all $t_0,t\in\mathbb{R}$, where $D_1\psi,D_2\psi\colon\mathbb{R}\times\mathbb{R}^{nN}\to\mathbb{R}$ are defined by
\begin{subequations}\label{eq:16112017:02}
\begin{align}
(D_1\psi)(s,p) & \ := \  - \sum_{m_1\in\Lambda}\widetilde{UV}_{m_1}(s)\,(X_{m_1}\psi)(p) \label{eq:16112017:02:A} \\
& \qquad - \sum_{m_1,m_2\in\Lambda}\widetilde{UV}_{m_2,m_1}(s)\,(X_{m_2}(X_{m_1}\psi))(p), \label{eq:16112017:02:B} \\
(D_2\psi)(s,p) & \ := \ \sum_{m_1,m_2,m_3\in\Lambda}u_{m_3}(s)\,\widetilde{UV}_{m_2,m_1}(s)\,(X_{m_3}(X_{m_2}(X_{m_1}\psi)))(p) \label{eq:16112017:02:C}
\end{align}
\end{subequations}
for all $(s,p)\in\mathbb{R}\times\mathbb{R}^{nN}$.
\end{proposition}
\begin{proof}
When we integrate the derivative of $\psi\circ\gamma\colon\mathbb{R}\to\mathbb{R}$, we obtain
\[
\psi(\gamma(t)) \ = \ \psi(\gamma(t_0)) + \sum_{m_1\in\Lambda}\int_{t_0}^{t}u_{m_1}(s)\,(X_{m_1}\psi)(\gamma(s))\,\mathrm{d}s,
\]
because~$\gamma$ is a solution of~\cref{eq:controlAffineForm}. We know from \Cref{thm:13112017}~\ref{thm:13112017:B} that each of the Lie derivatives~$X_{m_1}\psi$ is of class~$C^2$. Thus, we can apply integration by parts, which leads to
\begin{align*}
\psi(\gamma(t)) & \ = \ \psi(\gamma(t_0)) + \sum_{m_1,m_2\in\Lambda}\int_{t_0}^{t}u_{m_2}(s)\,\widetilde{UV}_{m_1}(s)\,(X_{m_2}(X_{m_1}\psi))(\gamma(s))\,\mathrm{d}s  \\
& \qquad + \sum_{m_1\in\Lambda}\widetilde{UV}_{m_1}(t_0)\,(X_{m_1}\psi)(\gamma(t_0)) - \sum_{m_1\in\Lambda}\widetilde{UV}_{m_1}(t)\,(X_{m_1}\psi)(\gamma(t))
\end{align*}
because of \Cref{thm:15112017}. Now we add and subtract $v_{m_2,m_1}(s)\,X_{m_2}(X_{m_1}\psi)(\gamma(s))$ in each of the above integrals. Note that by \Cref{thm:13112017}~\ref{thm:13112017:B}, the Lie derivatives $X_{m_2}(X_{m_1}\psi)$ are of class~$C^1$. Thus, we can apply again integration by parts and also \Cref{thm:15112017} to obtain
\begin{align*}
\psi(\gamma(t)) & \ = \ \psi(\gamma(t_0)) + \sum_{m_1,m_2\in\Lambda}\int_{t_0}^{t}v_{m_2,m_1}(s)\,X_{m_2}(X_{m_1}\psi)(\gamma(s))\mathrm{d}s \\
& \qquad  - (D_1\psi)(t_0,\gamma(t_0)) + (D_1\psi)(t,\gamma(t)) + \int_{t_0}^{t}(D_2\psi)(s,\gamma(s))\,\mathrm{d}s,
\end{align*}
where the functions $D_1\psi,D_2\psi\colon\mathbb{R}\times\mathbb{R}^{nN}\to\mathbb{R}$ are defined as in~\cref{eq:16112017:02}. The asserted equation~\cref{eq:16112017:01} now follows immediately from~\cref{eq:Duncan2}.
\end{proof}
\begin{remark}\label{thm:remainderContributions}
By \Cref{thm:14112017}, the averaged contribution~$Y\psi$ in~\cref{eq:16112017:01} is strictly negative as long as the gradient of the global potential function~$\psi$ is nonvanishing. This term leads to the desired effect that the value of~$\psi$ decreases along trajectories of~\cref{eq:controlAffineForm} if the remainder terms $D_1\psi,D_2\psi$ in~\cref{eq:16112017:02} are sufficiently small. The terms $D_1\psi,D_2\psi$ consist of the following two contributions:
\begin{enumerate}[label=(\Alph*)]
	\item\label{thm:remainderContributions:A} The time-varying functions $\widetilde{UV}_{m_1},\widetilde{UV}_{m_2,m_1},u_{m_3}\widetilde{UV}_{m_2,m_1}$. Suppose that the frequency coefficients~$\omega_{i,k}$ are given by~\cref{eq:omegaIK} in \Cref{exm:omegaIK}. We conclude from \Cref{thm:frequencyEstimates} that these functions converge uniformly to~$0$ when the global frequency parameter~$\omega$ tends to~$\infty$.
    \item\label{thm:remainderContributions:B} The Lie derivatives $X_{m_1}\psi$, $X_{m_2}(X_{m_1}\psi)$, and $X_{m_3}(X_{m_2}(X_{m_1}\psi))$. We conclude from \Cref{thm:13112017}~\ref{thm:13112017:B} that these functions become small when the agents are close to the set~\cref{eq:setOfDesiredStates} of target formations.
\end{enumerate}
The Lie derivatives in~\ref{thm:remainderContributions:B} ensure that the remainder terms $D_1\psi,D_2\psi$ vanish sufficiently fast when the value of the global potential function~$\psi$ approaches its optimal value $0$. Roughly speaking, this is the reason why \Cref{thm:mainResult} guarantees the existence of a small $r>0$ for which the sublevel set $\psi^{-1}(\leq{r})$ is in the domain of attraction. A large global frequency parameter~$\omega$ leads to the effect that the functions in~\ref{thm:remainderContributions:A} are small. This way one can ensure that $D_1\psi,D_2\psi$ remain sufficiently small in a larger sublevel set of~$\psi$.
Thus, when we increase~$\omega$, the influence of the averaged vector field~$Y$ dominates in a larger sublevel set of~$\psi$. This effect is also observed in the numerical simulations in \Cref{sec:simulations}.
\end{remark}

\subsection{Proof of \texorpdfstring{\Cref{thm:mainResult}}{Theorem~\ref{thm:mainResult}}}\label{sec:proofOfTheorem}
Recall that system~\cref{eq:kinematicPoints} under control~\cref{eq:controlLaw} can be written as the closed-loop system~\cref{eq:controlAffineForm}. We already know from \Cref{thm:ExistenceUniqueness} that there exists a unique global solution of~\cref{eq:controlAffineForm} for any initial condition.

Since we assume that for every element~$p$ of~\cref{eq:setOfDesiredStates}, the framework~$G(p)$ is infinitesimally rigid, \Cref{thm:11112017}~\ref{thm:11112017:C} ensures that there exist $c_{\psi},r_{\psi}>0$ such that $\|\nabla\psi(p)\|^2 \geq c_{\psi}\,\psi(p)$ for every $p\in\psi^{-1}(\leq{r_{\psi}})$. Because of \Cref{thm:14112017}, it follows that there exist $c_Y>0$ and $r_Y\in(0,r_{\psi})$ such that
\begin{equation}\label{eq:proof:01}
(Y\psi)(p) \ \leq \ -c_Y\,\psi(p)^2
\end{equation}
for every $p\in\psi^{-1}(\leq{r_Y})$. Now we take a look at the constituents of the functions $D_1\psi,D_2\psi\colon\mathbb{R}\times\mathbb{R}^{nN}\to\mathbb{R}$, which are defined in~\cref{eq:16112017:02}. It can be easily deduced from their definitions that the functions~$\widetilde{UV}_{m_1}$,  $\widetilde{UV}_{m_2,m_1}$, and~$u_{m_3}$ in~\cref{eq:16112017:02} are bounded. Moreover, we know from \Cref{thm:13112017}~\ref{thm:13112017:B} that the Lie derivatives of~$\psi$ along the~$X_{m}$ are bounded by multiples of certain powers of~$\psi$ on $\psi^{-1}(\leq{r_Y})$. This implies that there exist $c_1,c_2>0$ such that
\begin{subequations}\label{eq:proof:02}
\begin{align}
|(D_1\psi)(s,p)| & \ \leq \ c_1\,\psi(p)^{3/2}, \label{eq:proof:02:A} \\
|(D_2\psi)(s,p)| & \ \leq \ c_2\,\psi(p)^{5/2} \label{eq:proof:02:B}
\end{align}
\end{subequations}
for every $s\in\mathbb{R}$ and every $p\in\psi^{-1}(\leq{r_Y})$. We apply estimates~\cref{eq:proof:01,eq:proof:02} to~\cref{eq:16112017:01}, and obtain
\begin{align*}
\psi(\gamma(t)) & \ \leq \ \psi(\gamma(t_0)) + c_1\,\psi(\gamma(t_0))^{3/2} + c_1\,\psi(\gamma(t))^{3/2} \\
& \qquad - \int_{t_0}^{t}\big(c_{Y}\psi(\gamma(s))^2-c_2\,\psi(\gamma(s))^{5/2}\big)\,\mathrm{d}s
\end{align*}
for $t_0,t\in\mathbb{R}$ with $t>t_0$ if~$\gamma$ is a trajectory of~\cref{eq:controlAffineForm} such that $\psi(\gamma(s))\leq{r_Y}$ for every $s\in[t_0,t]$. We choose $r\in(0,r_Y/2)$ sufficiently small such that $1+c_1\,(2r)^{1/2}<2(1-c_1\,(2r)^{1/2})$ and such that $c:=(c_Y-c_2\,(2r)^{1/2})/2>0$. Then, we have
\begin{equation}\label{eq:proof:03}
\psi(\gamma(t)) \ \leq \ 2\,\psi(\gamma(t_0)) - 2\,c\int_{t_0}^{t}\,\psi(\gamma(s))^2\,\mathrm{d}s
\end{equation}
for $t_0,t\in\mathbb{R}$ with $t>t_0$ if~$\gamma$ is a trajectory of~\cref{eq:controlAffineForm} such that $\psi(\gamma(s))\leq{2r}$ for every $s\in[t_0,t]$. This implies that~\cref{eq:proof:03} holds in fact for every trajectory~$\gamma$ of~\cref{eq:controlAffineForm} and all $t_0,t\in\mathbb{R}$ with $t>t_0$ if $\psi(\gamma(t_0))\leq{r}$. It is now easy to see that the integral inequality~\cref{eq:proof:03} implies the asserted estimate~\cref{eq:speedOfConvergence}.

It is left to prove that the trajectories of~\cref{eq:controlAffineForm} with initial values in $\psi^{-1}(\leq{r})$ converge to some point of~\cref{eq:setOfDesiredStates}. For this purpose, fix a trajectory~$\gamma$ of~\cref{eq:controlAffineForm} with $\psi(\gamma(t_0))\leq{r}$ for some $t_0\in\mathbb{R}$. We already know from~\cref{eq:speedOfConvergence} that $\psi(\gamma(t))\leq{2r}$ for every $t>t_0$. We write~\cref{eq:controlAffineForm} as an integral equation and then we apply integration by parts on the right-hand side. Because of \Cref{thm:15112017}, this leads to
\begin{align*}
\gamma(t_2) & \ = \ \gamma(t_1) + \sum_{m_1,m_2\in\Lambda}\int_{t_1}^{t_2}u_{m_2}(s)\,\widetilde{UV}_{m_1}(s)\,\operatorname{D}\!X_{m_1}(\gamma(s))X_{m_2}(\gamma(s))\,\mathrm{d}s  \\
& \qquad + \sum_{m_1\in\Lambda}\widetilde{UV}_{m_1}(t_1)\,X_{m_1}(\gamma(t_1)) - \sum_{m_1\in\Lambda}\widetilde{UV}_{m_1}(t_2)\,X_{m_1}\gamma(t_2)
\end{align*}
for all $t_2,t_1\geq{t_0}$. It can be easily deduced from their definitions that the functions~$u_{m_2}$ and $\widetilde{UV}_{m_1}$ are bounded. Moreover, we know from \Cref{thm:13112017}~\ref{thm:13112017:A} that the maps~$X_{m_1}$ and $(\operatorname{D}\!X_{m_1})X_{m_2}$ are bounded by multiples of~$\psi$ and~$\psi^{3/2}$ on $\psi^{-1}(\leq{2r})$, respectively. Thus, there exist constants $c',c''>0$ such that
\[
\big\|\gamma(t_2)-\gamma(t_1)\big\| \ \leq \ c'\,\psi(\gamma(t_1)) + c'\,\psi(\gamma(t_2)) + c''\int_{t_1}^{t_2}\psi(\gamma(s))^{3/2}\,\mathrm{d}s
\]
for all $t_1,t_2\in\mathbb{R}$ with $t_2\geq{t_1}\geq{t_0}$. Now we apply estimate~\cref{eq:speedOfConvergence} and obtain
\[
\big\|\gamma(t_2)-\gamma(t_1)\big\| \ \leq \ \frac{4\,\psi(p_0)}{1+c\,\psi(p_0)\,(t_1-t_0)} + c''\int_{t_1}^{t_2}\Big(\frac{2\,\psi(p_0)}{1+c\,\psi(p_0)\,(s-t_0)}\Big)^{3/2}\,\mathrm{d}s
\] 
for all $t_1,t_2\in\mathbb{R}$ with $t_2\geq{t_1}\geq{t_0}$, where $p_0:=\gamma(t_0)$. This implies that for every $\varepsilon>0$, there exists $T>t_0$ such that $\big\|\gamma(t_2)-\gamma(t_1)\big\|\leq\varepsilon$ for all $t_2\geq{t_1}\geq{T}$. It follows that~$\gamma(t)$ converges to some $p\in\mathbb{R}^{nN}$ as $t\to\infty$. Since $\psi(\gamma(t))\to0$ as $t\to\infty$, we conclude that~$p$ is an element of~\cref{eq:setOfDesiredStates}.

%--------------------------------------------------------------------------------------------------------------
%--------------------------------------------------------------------------------------------------------------
%--------------------------------------------------------------------------------------------------------------

\section{Comparison to related approaches}\label{sec:discussion}
The aim of this section is to relate our approach to other known control strategies and to indicate how it can be extended to a more general situation. For the sake of simplicity, we restrict our discussion to a \emph{control-affine system} of the form
\begin{align}
\dot{p} & \ = \ \sum_{k=1}^{\mu}u_k\,B_k(p), \label{eq:generalCAF} \\
y & \ = \ \psi(p) \label{eq:generalOutput}
\end{align}
with smooth \emph{control vector fields} $B_1,\ldots,B_{\mu}\colon\mathbb{R}^n\to\mathbb{R}^n$, and a nonnegative smooth \emph{output function} $\psi\colon\mathbb{R}^{n}\to\mathbb{R}$. System \cref{eq:generalCAF} can be steered by specifying a control law for the real-valued \emph{input channels} $u_1,\ldots,u_{\mu}$. We assume that the nonnegative function~$\psi$ attains its smallest possible value $0$ at some point of~$\mathbb{R}^{n}$, i.e., the zero set $\psi^{-1}(0)\subseteq\mathbb{R}^{n}$ is not empty. In the context of formation control, one can interpret~\cref{eq:generalCAF} as the kinematic equations~\cref{eq:kinematicPoints} of a single agent who can only measure the current value~\cref{eq:generalOutput} of its individual potential function~\cref{eq:localPotential}. The current system state~$p(t)\in\mathbb{R}^{n}$ is treated as an unknown quantity. Our aim is to find time-varying output feedback that steers the system to the set of desired states $\psi^{-1}(0)$.

There are several ways to generalize the above situation. For instance, instead of a single system, one can consider a ``team'' of control-affine systems with individual output functions on a smooth manifold. One can also include an explicit time dependence of the control vector fields or a drift vector field which satisfies suitable boundedness conditions; cf.~\cite{Suttner2018}. Moreover, by imposing the assumption that the control vector fields and the output function have suitable invariance properties (such as translational invariance), it is also possible to treat the case in which $\psi^{-1}(0)$ is not necessarily compact. Our study of the formation control problem in the previous sections indicates how this can be done (cf. \Cref{thm:nonCompactness}). Since we want to keep the discussion brief and simple, we do not address these generalizations in the following.

The task of steering a dynamical system to a minimum of its output function based on real-time measurements of the output values, is extensively studied in the literature on \emph{extremum seeking control}. The reader is referred to~\cite{AriyurBook,ScheinkerBook,ZhangBook} for an overview. We show in the following paragraphs that the control law~\cref{eq:controlLaw} can be seen as a particular implementation of a more general strategy, which is also applied in the context of extremum seeking control; see, e.g.,~\cite{Duerr2013,Duerr2014,Duerr2015,Scheinker2013,Scheinker20132,Scheinker2014}. We explain the strategy by the example of system~\cref{eq:generalCAF} with output~\cref{eq:generalOutput}. Since we want to steer the system to the set of global minima of~$\psi$ it is certainly desirable to have information about descent directions of~$\psi$. Note that for every $k\in\{1,\ldots,{\mu}\}$ and every $p\in\mathbb{R}^{n}$, the vector $-(B_k\psi)(p)\,B_k(p)$ points into such a descent direction, where $(B_k\psi)(p)$ is the Lie derivative of~$\psi$ along~$B_k$ at~$p$; cf. \Cref{sec:basics}. Thus, the control law $u_k=-(B_k\psi)(p)$ for $k=1,\ldots,\mu$ would be a promising candidate for our purpose. Since we can only measure the values of~$\psi$ but not its derivative, this control law cannot be implemented directly. However, there is a way to circumvent this obstacle. A direct computation shows that the vector field $-(B_k\psi)\,B_k$ is equal to the Lie bracket of the vector fields~$\psi{B_k}$ and $B_k$, where $\psi{B_k}\colon\mathbb{R}^n\to\mathbb{R}^n$ is given by $(\psi{B_k})(p)=\psi(p)\,B_k(p)$. Note that the vector field~$\psi{B_k}$ only depends on~$\psi$ but not its derivative. This choice of the Lie bracket, which is due to~\cite{Duerr2013}, is not the only way to get access to $-(B_k\psi)\,B_k$. Another option, which appears in~\cite{Scheinker2014}, is the Lie bracket of the vector fields $(\sin\psi)B_k$ and $(\cos\psi)B_k$. More general, choose two functions $h_1,h_2\colon\mathbb{R}\to\mathbb{R}$, which are specified later, and define vector fields $X_{m}\colon\mathbb{R}^n\to\mathbb{R}^n$ as in~\cref{eq:newVectorFields} by
\[
X_{m}(p) \ := \ h_{\nu}(\psi(p))\,B_k(p)
\]
for every pair $m=(k,\nu)$ with $k\in\{1,\ldots,\mu\}$ and $\nu\in\{1,2\}$. Note that if~$h_1,h_2$ are differentiable at $y:=\psi(p)$ for some $p\in\mathbb{R}^{n}$, then we have
\[
[X_{(k,1)},X_{(k,2)}](p) \ = \ [h_1,h_2](y)\,(B_k\psi)(p)\,B_k(p),
\]
where $[h_1,h_2](y)$ is defined by~\cref{eq:hLie}. A systematic investigation on how~$h_1,h_2$ can be chosen such that $[X_{(k,1)},X_{(k,2)}]$ equals $-(B_k\psi)B_k$ is done in~\cite{Grushkovskaya20181}. As in \Cref{sec:proof}, we denote by~$\Lambda$ the set of all tuples $(k,\nu)$ with $k\in\{1,\ldots,\mu\}$ and $\nu\in\{1,2\}$.

So far we have only rewritten certain descent directions of~$\psi$ in terms of Lie brackets. However, it is not clear yet how system~\cref{eq:generalCAF} can be steered into these directions by means of output feedback. The idea is to use a suitable approximation of Lie Brackets. For this purpose, we choose for every $m\in\Lambda$ a family $(u_m^{\omega})_{\omega>0}$ of Lebesgue measurable and bounded functions $u_m^{\omega}\colon\mathbb{R}\to\mathbb{R}$, which are specified later. For every positive real number~$\omega$, we consider system~\cref{eq:generalCAF} under the control law
\[
u_{k} \ = \ u^{\omega}_{(k,1)}(t)\,h_1(\psi(p)) + u^{\omega}_{(k,2)}(t)\,h_2(\psi(p))
\]
for $k=1,\ldots,\mu$, which leads to the closed-loop system
\[
\Sigma^{\omega}\colon\qquad\dot{p} \ = \ \sum_{m\in\Lambda}u^{\omega}_{m}(t)\,X_m(p),
\]
cf.~\cref{eq:controlAffineForm}. We can interpret each~$\Sigma^{\omega}$ as a control-affine system with control vector fields~$X_m$ and open-loop controls~$u^{\omega}_m$. It is known from~\cite{Kurzweil1987,Kurzweil19882,Kurzweil1988,Sussmann1991,Liu19972,Liu1997} that if the vector fields~$X_m$ are of class~$C^1$, and if the families $(u_m^{\omega})_{\omega>0}$ satisfy certain averaging conditions in the limit $\omega\to\infty$, then, for any fixed initial condition $(t_0,p_0)$, the trajectories of the systems~$\Sigma^{\omega}$ converge on a compact interval in the limit $\omega\to\infty$ to the trajectory of
\[
\Sigma^{\infty}\colon\qquad\dot{p} \ = \ Y(p) \ := \ \frac{1}{2}\sum_{k=1}^{\mu}[X_{(k,1)},X_{(k,2)}](p)
\]
with initial condition $(t_0,p_0)$. Note that $Y\colon\mathbb{R}^n\to\mathbb{R}^n$ corresponds to the vector field in~\cref{eq:Duncan0}. The convergence property of trajectories holds, if the functions~$h_1,h_2$ are of class~$C^1$ and if we let
\begin{align*}
u^{\omega}_{(k,1)}(t) & \ := \ \sqrt{\omega\,\Omega_k}\cos(\omega\,\Omega_k\,t + \varphi_k), \\
u^{\omega}_{(k,2)}(t) & \ := \ \sqrt{\omega\,\Omega_k}\sin(\omega\,\Omega_k\,t + \varphi_k),
\end{align*}
for $k=1,\ldots,\mu$, where $\Omega_1,\ldots,\Omega_{\mu}>0$ are pairwise distinct positive real numbers, and $\varphi_1,\ldots,\varphi_{\mu}\in\mathbb{R}$ are arbitrary. Note that we use the same trigonometric functions in \Cref{sec:controlLaw}. The averaging conditions that we mentioned earlier are indicated in \Cref{thm:frequencyEstimates,thm:15112017}. The general theory is presented in~\cite{Liu19972,Liu1997}, where the frequency parameter~$\omega$ is treated as a sequence index~$j$.

Assume that we have chosen the functions~$h_1,h_2$ in a suitable way so that the set of desired states $\psi^{-1}(0)$ is locally asymptotically stable for~$\Sigma^{\infty}$. Under suitable averaging assumptions on the families $(u_m^{\omega})_{\omega>0}$ in the limit $\omega\to\infty$ and also smoothness assumptions on the vector fields~$X_m$, it is shown in~\cite{Duerr2013} that the convergence of trajectories is in fact uniform with respect to the initial time and also uniform with respect to the initial state within compact sets. This stronger notion of convergence of trajectories ensures that the set of desired states $\psi^{-1}(0)$ becomes \emph{practically locally uniformly asymptotically stable} for~$\Sigma^{\omega}$ if~$\omega$ is chosen sufficiently large. The word \emph{uniform} refers to uniformity with respect to the time parameter. Moreover, \emph{practically} means that the trajectories of~$\Sigma^{\omega}$ are only attracted by a neighborhood of $\psi^{-1}(0)$ but not by $\psi^{-1}(0)$ itself. However, it is not known how large the frequency parameter~$\omega$ has to be chosen to ensure practical stability.

The proof of practical stability for~$\Sigma^{\omega}$ in~\cite{Duerr2013} is based on a suitable averaging analysis, which leads to a similar integral equation as~\cref{eq:16112017:01} in \Cref{thm:16112017}. This integral equation also contains the averaged vector field~$Y$ of Lie brackets and two time-varying remainder vector fields $D_1^{\omega}$ and $D_2^{\omega}$, which additionally depend on the frequency parameter $\omega>0$. When~$\omega$ tends to~$\infty$, the vector fields $D_1^{\omega},D_2^{\omega}$ vanish and only~$Y$ remains. This roughly explains why local asymptotic stability of~$\Sigma^{\infty}$ induces practical local asymptotic stability of~$\Sigma^{\omega}$ when~$\omega$ is sufficiently large. The same effect for large~$\omega$ is also discussed in \Cref{thm:remainderContributions}. Note that a large frequency parameter~$\omega$ alone only leads to practical local asymptotic stability. To obtain the full notion of local asymptotic stability for~$\Sigma^{\omega}$, it is also necessary to ensure that the remainders $D_1^{\omega},D_2^{\omega}$ vanish sufficiently fast when the system state approaches the set $\psi^{-1}(0)$ of desired states. In the present paper, we derive the corresponding boundedness properties in \Cref{sec:boundednessLemmas}. A similar approach can be found in~\cite{Grushkovskaya20181,Suttner2018}. However, the results in~\cite{Grushkovskaya20181,Suttner2018} only ensure local asymptotic stability if $\omega>0$ is sufficiently large. Our main result, \Cref{thm:mainResult}, guarantees local asymptotic stability with a possibly small domain of attraction even if the frequencies are small. The domain of attraction increases if we choose large frequencies, since this leads to smaller remainders $D_1^{\omega},D_2^{\omega}$, cf. \Cref{thm:remainderContributions}. Finally, it is worth to mention that similar results also appear in~\cite{Moreau2003,Morin1999} for the stabilization of homogeneous systems. They also rely on a combination of averaging and suitable boundedness properties of the vector fields and their derivatives.

We return to system~\cref{eq:generalCAF} with output~\cref{eq:generalOutput}. Let $h_1,h_2\colon\mathbb{R}\to\mathbb{R}$ be two functions with the properties \ref{def:hNu:i}-\ref{def:hNu:vi} in \Cref{sec:controlLaw}. Let $\omega_1,\ldots,\omega_{\mu}$ be pairwise distinct positive real constants, and let $\varphi_1,\ldots,\varphi_{\mu}\in\mathbb{R}$. For $k=1,\ldots,\mu$, define $u_{(k,1)},u_{(k,2)}\colon\mathbb{R}\to\mathbb{R}$ by
\begin{align*}
u_{(k,1)}(t) & \ := \ \sqrt{\omega_{k}}\,\cos(\omega_{k}t + \varphi_k), \\
u_{(k,2)}(t) & \ := \ \sqrt{\omega_{k}}\,\sin(\omega_{k}t + \varphi_k).
\end{align*}
Following~\cref{eq:controlLaw}, we propose the output-feedback control law
\begin{equation}\label{eq:ESC:law}
u_k \ = \ u_{k,1}(t)\,h_{1}(\psi(p)) + u_{k,1}(t)\,h_{2}(\psi(p))
\end{equation}
for $k=1,\ldots,\mu$ to steer \cref{eq:generalCAF} to a minimum of~$\psi$. We remark that an implementation of~\cref{eq:ESC:law} requires no other information than real-time measurements of the output~\cref{eq:generalOutput}.  The same argument as in the proof of \Cref{thm:12112017}~\ref{thm:12112017:A} shows that the functions $h_{\nu}\circ\psi$, with $\nu=1,2$, are of class~$C^1$. This ensures that~\cref{eq:generalCAF} under~\cref{eq:ESC:law} has a unique maximal solution for every initial condition. For every $r>0$, define the sublevel set
\[
\psi^{-1}(\leq{r}) \ := \ \{p\in\mathbb{R}^{n} \ | \ \psi(p)\leq{r}\}.
\]
Following the analysis in \Cref{sec:proof}, it is now easy to derive the following result.
\begin{theorem}\label{thm:ESC}
Assume that there exists $p^{\ast}\in\mathbb{R}^n$ such that the following conditions are satisfied:
\begin{enumerate}[label=(\roman*)]
	\item\label{thm:ESC:i} $\psi(p^{\ast})=0$ is a strict local minimum of~$\psi$ and the second derivative of~$\psi$ at~$p^{\ast}$ is positive definite;
    \item\label{thm:ESC:ii} there exists a neighborhood $W\subseteq\mathbb{R}^n$ of~$p^{\ast}$ such that for every $p\in{W}$, the vectors $B_1(p),\ldots,B_{\mu}(p)$ span~$\mathbb{R}^n$.
\end{enumerate}
Then, there exist constants $c,r>0$ such that for every $t_0\in\mathbb{R}$, and every $p_0\in\mathbb{R}^n$ in the connected component of $\psi^{-1}(\leq{r})$ containing~$p^{\ast}$, the maximal solution~$\gamma$ of system~\cref{eq:generalCAF} under the control law~\cref{eq:ESC:law} with initial condition $\gamma(t_0)=p_0$ exists on $[t_0,\infty)$, and~$\gamma(t)$ converges to~$p^{\ast}$ as $t\to\infty$ with
\begin{equation}\label{thm:ESC:eq:01}
\psi(\gamma(t)) \ \leq \ \frac{2\,\psi(p_0)}{1+c\,\psi(p_0)\,(t-t_0)}
\end{equation}
for every $t\geq{t_0}$.
\end{theorem}
Note that the assumption of infinitesimal rigidity of the target formations in \Cref{thm:mainResult} is replaced in \Cref{thm:ESC} by assumption~\ref{thm:ESC:i}. Because of \Cref{thm:10112017}, this assumption ensures that estimate~\cref{eq:11112017:C} in \Cref{thm:11112017}~\ref{thm:11112017:C} is satisfied for the output function~$\psi$ in a neighborhood of~$p^{\ast}$. In the context of formation control, the velocity directions~$b_{i,k}$ of the agents in~\cref{eq:kinematicPoints} span the entire Euclidean space at any point. This property is locally ensured in \Cref{thm:ESC} by assumption~\ref{thm:ESC:ii}.

We remark that \Cref{thm:ESC} assumes that the set of desired states consists only of a single point~$p^{\ast}$. The result can be extended to a possibly noncompact set of desired states if the control vector fields $B_1,\ldots,B_{\mu}$ and the output function~$\psi$ have suitable invariance properties. For example, for point agents in the Euclidean space, we have invariance under the action of the Euclidean group, which reduces the set of target formations to finitely many orbits. The analysis in~\Cref{sec:proof} also indicates how \Cref{thm:ESC} can be extended to multiple control systems with individual output functions.

As explained in \Cref{thm:remainderContributions}, the magnitude $r>0$ of the sublevel  $\psi^{-1}(\leq{r})$ depends on the choice of the frequency coefficients $\omega_1,\ldots,\omega_{\mu}$. Under suitable assumptions, it is also possible to extend \Cref{thm:ESC} from a local to a semi-global stability result. For this purpose, assumption~\ref{thm:ESC:i} has to be replaced by the conditions that~$p^{\ast}$ is a strict global minimum of~$\psi$, that the second derivative of~$\psi$ at~$p^{\ast}$ is positive definite, and that~$\psi$ has no other stationary points than~$p^{\ast}$. Assumption~\ref{thm:ESC:ii} has to be replaced by the condition that the vectors $B_1(p),\ldots,B_{\mu}(p)$ span~$\mathbb{R}^n$ for every $p\in\mathbb{R}^n$. Finally, in addition to the properties \ref{def:hNu:i}-\ref{def:hNu:vi} in \Cref{sec:controlLaw}, one has to ensure that $[h_1,h_2](y)<0$ holds for every $y>0$. For instance, this is satisfied if~$h_1,h_2$ are chosen as in \Cref{exm:hNu}. Then, for every compact neighborhood~$K_0$ of~$p^{\ast}$ in~$\mathbb{R}^n$, one can find sufficiently large frequencies $\omega_1,\ldots,\omega_{\mu}$ such that~$K_0$ is uniformly asymptotically stable for system~\cref{eq:generalCAF} under the control law~\cref{eq:ESC:law} with~$K_0$ in the domain of attraction.

Finally, we compare \Cref{thm:ESC} to the results in the studies on extremum seeking control by Lie bracket approximations that we cited earlier in this section. The main advantage of \Cref{thm:ESC} is that local uniform asymptotic stability can be obtained even if the pairwise distinct frequencies $\omega_k>0$, $k=1,\ldots,\mu$, are arbitrarily small. So far, the results in the literature only ensure (practical) asymptotic stability if the frequencies~$\omega_k$ as well as their distances $|\omega_l-\omega_k|$ are chosen sufficiently large. In the context of extremum seeking, the control vector fields as well as the output function are treated as unknown quantities. Only real-time measurements of the output~\cref{eq:generalOutput} are available. For such a situation, there is no known rule how to obtain suitable values for the~$\omega_k$. \Cref{thm:ESC} solves this problem in the sense that local uniform asymptotic stability is obtained independent of the choice of the~$\omega_k$. As explained in the previous paragraph, it is also possible to obtain semi-global uniform asymptotic stability for system~\cref{eq:generalCAF} under the control law~\cref{eq:ESC:law}. Unlike many other similar approaches, control law~\cref{eq:ESC:law} can lead to convergence to~$p^{\ast}$ and not only to convergence to an unknown neighborhood of~$p^{\ast}$. Another advantage compared to other studies is the flexibility in the choice of the frequencies. We do not assume that the~$\omega_k$ are rational multiples of each other. It already suffices that they are pairwise distinct.

%--------------------------------------------------------------------------------------------------------------
%--------------------------------------------------------------------------------------------------------------
%--------------------------------------------------------------------------------------------------------------

\section{Conclusions and future work}\label{sec:conclusions}
We have shown that distance measurements provide enough information to locally stabilize infinitesimally rigid target formations in the Euclidean space of arbitrary dimension. The proposed control law is distributed, and its implementation requires only the currently sensed distances. Certainly, a disadvantage compared to the well-established gradient-based control law is the relatively small domain of attraction for small frequency coefficients. On the other hand, our feedback law can lead to a closed-loop system without undesired equilibria. A promising direction for future research might be a suitable superposition of both control laws. This, perhaps, could lead to global asymptotic stability. There are several other potential applications for the proposed control strategy in the field of multi-agent systems. Many distributed coordination algorithms involve potential functions of inter-agent distances such as distributed navigation~\cite{Leonard2001}, swarming~\cite{Dimarogonas2008} and flocking~\cite{Olfati2006}. The implementation is usually derived from a distributed gradient vector field of a potential function, which often requires relative position measurements. Our approach can also be applied to these coordination control tasks, and allows an implementation if only distance measurements are available.

\bibliographystyle{abbrv}
\bibliography{bibFile}
\end{document}